\documentclass{article}
\usepackage{amsfonts,amsmath,amssymb,amsthm,cite,enumerate,graphicx,latexsym,mathrsfs}
\usepackage{authblk}
\usepackage{hyperref}
\usepackage[T1]{fontenc}
\usepackage[utf8]{inputenc}
\usepackage{float}
\usepackage{geometry}
\allowdisplaybreaks[2]

\date{}
\makeatletter

\newcommand{\Rmnum}[1]{\expandafter\@slowromancap\romannumeral #1@}
\makeatother

\numberwithin{equation}{section}
\newtheorem{corollary}{Corollary}[section]

\newtheorem{lemma}{Lemma}[section]
\newtheorem{proposition}{Proposition}[section]
\newtheorem{remark}{Remark}[section]

\newtheorem{theorem}{Theorem}[section]

\newcommand\theref[1]{Theorem~\ref{#1}}
\newcommand\lemref[1]{Lemma~\ref{#1}}

\newcommand\proref[1]{Proposition~\ref{#1}}

\newcommand\secref[1]{Section~\ref{#1}}

\title{Subsonic time-periodic solution to damped compressible Euler equations with large entropy}
\author[a]{Peng Qu\thanks{{e}-mail: pqu@fudan.edu.cn}}
\author[b]{Huimin Yu\thanks{Corresponding author {e}-mail: hmyu@sdnu.edu.cn}}
\author[b]{Xiaomin Zhang\thanks{{e}-mail: zxm15924687@163.com}}
\affil[a]{School of Mathematical Sciences, Shanghai Key Laboratory for Contemporary Applied Mathematics, Fudan  University, Shanghai 200433 China}
\affil[b]{Department of mathematics, Shandong Normal University, Jinan 250014 China}
\begin{document}
\begin{sloppypar}
\date{}
\maketitle
\begin{center}
\begin{minipage}{130mm}{\small
\textbf{Abstract}:
In this paper, one-dimensional nonisentropic compressible Euler equations with linear damping $\alpha(x)\rho u$ are analyzed.~We want to explore the conditions under which a subsonic temporal periodic boundary can trigger a time-periodic $C^{1}$ solution. To achieve this aim, we use a technically constructed iteration scheme and give the sufficient conditions to guarantee the existence, uniqueness and stability of the $C^{1}$ time-periodic solutions on the perturbation of a subsonic Fanno flow.~It is worthy to be pointed out that the entropy exhibits large amplitude under the assumption that the inflow sound speed is small.~However, it is crucial to assume that the boundary conditions possess a kind of dissipative structure at least on one side, which is used to cancel the nonlinear accelerating effect in the system.~The results indicate that the time-periodic feedback boundary control with dissipation can stabilize the nonisentropic compressible Euler equations around the Fanno flows.
}
\end{minipage}
\end{center}
\section{Introduction}

\indent\indent The compressible Euler equations are a set of partial differential equations which can be used to analyze a wide range of fluid dynamics problems, such as the flow of gases in turbomachinery, the behavior of shock waves, and the development of high-speed jets. They are of importance in understanding and predicting the motion of compressible fluids in various engineering and scientific applications. The compressible Euler equations with damping are the modified version of the standard compressible Euler equations and include an additional term to model the dissipation of energy arising from friction (see~\cite{Dafermos, DafermosC}) or other dissipative processes. The presence of damping is important for accurately modeling real-world fluid dynamics. Typically, the damping term is represented as a proportionality factor multiplying the right-hand side of the system. In this paper,we study the nonisentropic compressible Euler equations with linear damping $\alpha(x)\rho u$ in the Eulerian coordinates
\begin{equation}\label{a1}
\left\{\begin{aligned}
&\rho_{t}+(\rho u)_{x}=0,\\
&(\rho u)_{t}+(\rho u^{2}+p)_{x}=\alpha(x)\rho u,\\
&S_{t}+uS_{x}=0,
\end{aligned}\right.
\end{equation}
where $(t,x)\in\{(t,x)|t\in\mathbb{R}_{+},x\in[0,L]\}$, $\rho,u,S$ are the density, the velocity, the entropy respectively, and $p=p(\rho,S)$ is the pressure. The general pressure law we considered in this paper is
$$
p=p(\rho,S)=\rho^{\gamma}\varphi(S),
$$
where $\varphi(S)$ is a positive smooth function, the adiabatic gas exponent $1<\gamma<3$ and pressure coefficient is normalized to $1$. The damping coefficient $\alpha(x)$ is a non-positive $C^{1}$ smooth function satisfying
\begin{align}
&\alpha_{*}\leq\alpha(x)\leq0,\label{a2}
\end{align}
where $\alpha_{*}<0$ is a constant.

For nonisentropic flows, there are many works on the existence of $C^{1}$ solution under the hypothesis of the small initial data. For example, \cite{Xu} studied the asymptotic stability of smooth solutions to the multidimensional nonisentropic hydrodynamic model for semiconductors under the assumption that the initial data are a small perturbation of the stationary solutions for the thermal equilibrium state; \cite{Zhangy} investigated the global existence and asymptotic behavior of classical solutions for the 3D compressible non-isentropic damped Euler equations on a periodic domain when the initial data is near an equilibrium; \cite{Wangw} explored the global existence of classical solution to the Cauchy problem for the nonisentropic Euler–Maxwell system with a nonconstant background density when the initial perturbation is sufficiently small; \cite{Lit} analyzed the global well-posedness and large-time behavior of smooth solutions to the one-dimensional non-isentropic compressible Euler equations with relaxation when the initial datum is sufficiently close to a constant equilibrium state. Besides, the authors also analyzed finite-time blowup of classical solutions under general conditions on initial data. For the other works about the nonisentropic system with the Poisson term, readers refer to~\cite{Liy, Ruan, Hong}.

Time-periodic solutions play a crucial role in understanding the intrinsic mechanisms of natural phenomena and designing efficient algorithms for scientific and engineering problems. The exploration of these solutions has led to the development of numerous mathematical models and computational methods, which in turn have advanced our understanding of complex systems.
For further insights into this area, readers are encouraged to explore publications such as \cite{Feireisl, Hmidi, Kagei, Ji, Ma, Wang} and other related references therein. Among these works, many are driven by periodic external forces. However, in this paper, we concern about time-periodic solutions triggered by time-periodic boundary, which is of great significance. Some researchers have contributed to understanding various aspects of this subject. For instance, \cite{Yuan} investigated the existence of time-periodic supersonic solutions to the one-dimensional isentropic compressible Euler equations subjected to time-periodic boundary conditions within a finite interval, after a certain start-up time that depends on the interval's length. \cite{Mas} explored the global existence and stability of temporal periodic solutions to one-dimensional compressible non-isentropic Euler equations with the source term $\beta\rho|u|^{\alpha}u$ under small perturbations around the steady supersonic flow, wherein the boundary conditions are temporal periodic. Additionally, \cite{Qu} examined the existence, uniqueness and stability of quasilinear hyperbolic systems incorporating time-periodic dissipative boundary conditions. Furthermore, \cite{Fang} analyzed the well-posedness of the time-periodic classical solutions to the general nonhomogeneous quasilinear hyperbolic equations featuring a  weak diagonal dominant structure. In particular, \cite{Qup} devoted to the existence, stability and regularity of the subsonic time-periodic solutions to the one-dimensional isentropic compressible Euler equations with linear damping. For other more results, we refer to~\cite{Yuh, Zhangx}.

It is worth noting that the aforementioned literatures predominantly focused on investigations under small perturbations of initial data or boundary conditions.~In contrast, our research contemplates time-periodic solutions of damped nonisentropic compressible Euler equations under large amplitude of the entropy $S$ at the initial $t=0$ and the boundary $x=0$. Follow the idea used in~\cite{Qup}, where the background solutions are the constant state $(\tilde{\rho},\tilde{u})=(\underline{\rho},0)$, the non-constant background solutions lead the reduced system~\eqref{b8}-\eqref{b10} to be a hyperbolic balance laws without dissipation or even with some accelerating growth effect. Fortunately, by an ingeniously constructed iterative scheme, we cancel the accelerating growth effect by the boundary dissipation when the background sound speed is small enough. In addition, the system~\eqref{b8}-\eqref{b10} can cause derivative loss, which make certain difficulties for our later a prior estimates. However, we find the entropy is balance along the $2$-nd characteristic curve in~\eqref{b9} and then we skillfully transform the derivative of entropy over $x$ into the derivative along the $1$-st or $3$-rd characteristic curve direction, ultimately overcoming these difficulties. %H. Liu~\cite{Liu} studied the existence of global smooth solution to the Cauchy problem for nonisentropic gas dynamics systems under the condition that the $C^{0}$ norm of the initial data is large and $C^{1}$ norm of the initial data is suitably small.

The plan of the rest of this article is as follows.  In~\secref{s2}, we reformulate the original
system to obtain a quasi-linear hyperbolic system and give some basic facts, which
will be used later. The main results (i.e.~\theref{t1}-\theref{t5}) are also stated here. In~\secref{s3}, the existence of time-periodic solutions is proved by the linearized iteration method. The stability and uniqueness of the time-periodic solution are discussed in~\secref{s4}. The~\secref{s5} and~\secref{s6} are devoted to the higher regularity and stability of the time-periodic solution respectively, if we assume the boundary has a higher regularity.

\section{Preliminaries and Main Results}\label{s2}
\indent\indent We first give some facts on the steady state(i.e.Fanno flows), which are controlled by the following ODEs:
\begin{align}\label{A5}
\left\{
\begin{aligned}
&\tilde{\rho}'\tilde{u}+\tilde{\rho}\tilde{u}'=0,\\
&\tilde{u}\tilde{u}'+\gamma\tilde{\rho}^{\gamma-2}\varphi(\tilde{S})\tilde{\rho}'
=\alpha\tilde{u},\\
&\tilde{u}\tilde{S}'=0,\\
&(\tilde{\rho},\tilde{u},\tilde{S})^{\top}|_{x=0}=(\rho_{-},u_{-},S_{-})^{\top}
\end{aligned}
\right.
\end{align}
or equivalently
\begin{align}\label{a3}
\left\{
\begin{aligned}
&\tilde{c}'\tilde{u}+\frac{\gamma-1}{2}\tilde{c}\tilde{u}'=0,\\
&\tilde{u}\tilde{u}'+\frac{2}{\gamma-1}\tilde{c}\tilde{c}'
=\alpha\tilde{u},\\
&\tilde{u}\tilde{S}'=0,\\
&(\tilde{u},\tilde{c},\tilde{S})^{\top}|_{x=0}=(u_{-},c_{-},S_{-})^{\top},
\end{aligned}
\right.
\end{align}
where $\tilde{c}=\sqrt{\gamma}\tilde{\rho}^{\frac{\gamma-1}{2}}\sqrt{\varphi(\tilde{S})}$ and $\rho_{-}, u_{-}, c_{-},S_{-}$ are positive constants satisfying $u_{-}<c_{-}=\sqrt{\gamma}\rho_{-}^{\frac{\gamma-1}{2}}\sqrt{\varphi(S_{-})}$.
The paper~\cite{Yuh} has already discussed both the supersonic and subsonic Fanno flows for isentropic Euler equations clearly. For the non-constant entropy case, noting the equation~$\eqref{A5}_{3}$ or $\eqref{a3}_{3}$, which implies that the entropy for the steady flow must be constant in the duct, we still get the following result similar to Lemma~$2.1$ in~\cite{Yuh}
\begin{lemma}\label{L1}
If $0<u_{-}<c_{-}$ and the duct length $L< L_M$, where the maximum allowance duct length $L_M$ is a positive constant only depending on $\alpha, \gamma, c_-$, and $u_-$, then the Cauchy problem~\eqref{a3} admits a unique smooth positive solution $(\tilde{u}(x),\tilde{c}(x),\tilde{S}(x))^{\top}$ which satisfies the following properties:
\begin{align}
0<u_-<\tilde{u}(x)<\tilde{c}(x)<c_-, \quad \tilde{S}(x)=S_{-}.\label{A3}
\end{align}
\end{lemma}
For the nonisentropic compressible Euler equations, we compute that its eigenvalues are
$$
\lambda_{1}=u-c,\quad \lambda_{2}=u,\quad \lambda_{3}=u+c,
$$
where $c=\sqrt{p_{\rho}}=\sqrt{\gamma}\rho^{\frac{\gamma-1}{2}}\sqrt{\varphi(S)}$.

Suppose $(\tilde{u},\tilde{c},\tilde{S})^{\top}$ is the steady subsonic Fanno flows given in~\lemref{L1}, then we have
\begin{align*}
\lambda_{1}(\tilde{u},\tilde{c})<0<\lambda_{2}(\tilde{u},\tilde{c})<\lambda_{3}(\tilde{u},\tilde{c}),
\end{align*}
and
\begin{align}
\lambda_{1}(u,c)<0<\lambda_{2}(u,c)<\lambda_{3}(u,c),\quad (u,c)\in\wp\label{b1}
\end{align}
for a small neighborhood $\wp$ of $(\tilde{u},\tilde{c})$.

Using the Riemann invariants
\begin{align}
r_{1}=u-\frac{2}{\gamma-1}c,\quad r_{2}=S,\quad r_{3}=u+\frac{2}{\gamma-1}c,\label{b2}
\end{align}
we write the equations~\eqref{a1} into the following diagonal form
\begin{align}\label{b3}
\left\{
\begin{aligned}
&\frac{\partial r_{1}}{\partial t}+\lambda_{1}\frac{\partial r_{1}}{\partial x}=\frac{\alpha(x)}{2}(r_{1}+r_{3})+\frac{(\gamma-1)\varphi'(r_{2})}{16\gamma\varphi(r_{2})}(r_{3}-r_{1})^{2}\frac{\partial r_{2}}{\partial x},\\
&\frac{\partial r_{2}}{\partial t}+\lambda_{2}\frac{\partial r_{2}}{\partial x}=0,\\
&\frac{\partial r_{3}}{\partial t}+\lambda_{3}\frac{\partial r_{3}}{\partial x}=\frac{\alpha(x)}{2}(r_{1}+r_{3})+\frac{(\gamma-1)\varphi'(r_{2})}{16\gamma\varphi(r_{2})}(r_{3}-r_{1})^{2}\frac{\partial r_{2}}{\partial x},
\end{aligned}\right.
\end{align}
where $\lambda_{1}=\frac{\gamma+1}{4}r_{1}+\frac{3-\gamma}{4}r_{3},~ \lambda_{2}=\frac{r_{1}+r_{3}}{2},~ \lambda_{3}=\frac{3-\gamma}{4}r_{1}+\frac{\gamma+1}{4}r_{3}$.

In the same way, we also write the steady equations~${\eqref{a3}}_{1}-\eqref{a3}_{3}$ into the equations of Riemann invariants
\begin{align}\label{b7}
\left\{
\begin{aligned}
&\tilde{\lambda}_{1}\tilde{r}'_{1}=\frac{\alpha(x)}{2}(\tilde{r}_{1}+\tilde{r}_{3}),\\
&\tilde{\lambda}_{3}\tilde{r}'_{3}=\frac{\alpha(x)}{2}(\tilde{r}_{1}+\tilde{r}_{3}),\\
&\tilde{r}_{2}=const.,
\end{aligned}
\right.
\end{align}
where $$\tilde{r}_{1}=\tilde{u}-\frac{2}{\gamma-1}\tilde{c},~~\tilde{r}_{2}=\tilde{S},~~\tilde{r}_{3}=\tilde{u}+\frac{2}{\gamma-1}\tilde{c},$$ $$\tilde{\lambda}_{1}=\frac{\gamma+1}{4}\tilde{r}_{1}+\frac{3-\gamma}{4}\tilde{r}_{3},~~ \tilde{\lambda}_{3}=\frac{3-\gamma}{4}\tilde{r}_{1}+\frac{\gamma+1}{4}\tilde{r}_{3}.$$
We further assume that the initial data and boundary conditions of equations~\eqref{b3} are
\begin{align}
t=0:~~ &r_{1}(0,x)={r}_{1_{0}}(x),\quad r_{2}(0,x)={r}_{2_{0}}(x),\quad r_{3}(0,x)={r}_{3_{0}}(x),\label{b4}\\
x=0:~~ &r_{2}(t,0)=G_{2}(t)+K_{2}(r_{1}(t,0)-\tilde{r}_{1}(0)),\label{B5}\\
&r_{3}(t,0)=G_{3}(t)+K_{3}(r_{1}(t,0)-\tilde{r}_{1}(0)),\label{b5}\\
x=L:~~ &r_{1}(t,L)=G_{1}(t)+K_{1}(r_{3}(t,L)-\tilde{r}_{3}(L)),\label{b6}
\end{align}
where $K_{2}$ is a given constant, $|K_{1}|\leq1, |K_{3}|\leq1, |K_{1}K_{3}|<1$ and $G_{i}(t)(i=1,2,3)$ are $C^{1}$ time-periodic functions with the period $P>0$, i.e.
$$
G_{i}(t+P)=G_{i}(t), i=1,2,3.
$$

Let
$$\Phi=(\Phi_{1},\Phi_{2},\Phi_{3})^{\top}=(r_{1}-\tilde{r}_{1},r_{2}-\tilde{r}_{2},r_{3}-\tilde{r}_{3})^{\top},\quad
\tilde{\Phi}=(\tilde{\Phi}_{1},\tilde{\Phi}_{2},\tilde{\Phi}_{3})^{\top}=(\tilde{r}_{1},\tilde{r}_{2},\tilde{r}_{3})^{\top},$$
then we use~\eqref{b3} and~\eqref{b7} to write the equations of $\Phi$ as follows
\begin{align}
\frac{\partial\Phi_{1}}{\partial t}+\lambda_{1}(\Phi+\tilde{\Phi})\frac{\partial\Phi_{1}}{\partial x}
=&\frac{\alpha}{2}(1-\frac{(\gamma+1)(\tilde{\Phi}_{1}+\tilde{\Phi}_{3})}{(\gamma+1)\tilde{\Phi}_{1}+(3-\gamma)\tilde{\Phi}_{3}})\Phi_{1}
+\frac{\alpha}{2}(1-\frac{(3-\gamma)(\tilde{\Phi}_{1}+\tilde{\Phi}_{3})}{(\gamma+1)\tilde{\Phi}_{1}+(3-\gamma)\tilde{\Phi}_{3}})\Phi_{3}\notag\\
&+\frac{(\gamma-1)\varphi'(\Phi_{2}+\tilde{\Phi}_{2})}{16\gamma\varphi(\Phi_{2}+\tilde{\Phi}_{2})}(\Phi_{3}+\tilde{\Phi}_{3}-\Phi_{1}-\tilde{\Phi}_{1})^{2}\frac{\partial\Phi_{2}}{\partial x},\label{b8}\\
\frac{\partial\Phi_{2}}{\partial t}+\lambda_{2}(\Phi+\tilde{\Phi})\frac{\partial\Phi_{2}}{\partial x}=&0,\label{b9}\\
\frac{\partial\Phi_{3}}{\partial t}+\lambda_{3}(\Phi+\tilde{\Phi})\frac{\partial\Phi_{3}}{\partial x}
=&\frac{\alpha}{2}(1-\frac{(\gamma+1)(\tilde{\Phi}_{1}+\tilde{\Phi}_{3})}{(3-\gamma)\tilde{\Phi}_{1}+(\gamma+1)\tilde{\Phi}_{3}})\Phi_{3}
+\frac{\alpha}{2}(1-\frac{(3-\gamma)(\tilde{\Phi}_{1}+\tilde{\Phi}_{3})}{(3-\gamma)\tilde{\Phi}_{1}+(\gamma+1)\tilde{\Phi}_{3}})\Phi_{1}\notag\\
&+\frac{(\gamma-1)\varphi'(\Phi_{2}+\tilde{\Phi}_{2})}{16\gamma\varphi(\Phi_{2}+\tilde{\Phi}_{2})}(\Phi_{3}+\tilde{\Phi}_{3}-\Phi_{1}-\tilde{\Phi}_{1})^{2}\frac{\partial\Phi_{2}}{\partial x}.\label{b10}
\end{align}
The corresponding initial data and boundary conditions are
\begin{align}
t=0:~~ &\Phi_{1}(0,x)={\Phi}_{1_{0}}(x)={r}_{1_{0}}(x)-\tilde{r}_{1}(x),\quad \Phi_{2}(0,x)={\Phi}_{2_{0}}(x)={r}_{2_{0}}(x)-\tilde{r}_{2}(x),\notag\\
&\Phi_{3}(0,x)={\Phi}_{3_{0}}(x)={r}_{3_{0}}(x)-\tilde{r}_{3}(x),\label{b11}\\
x=0:~~ &\Phi_{2}(t,0)=H_{2}(t)+K_{2}\Phi_{1}(t,0),\quad
\Phi_{3}(t,0)=H_{3}(t)+K_{3}\Phi_{1}(t,0),\label{b12}\\
x=L:~~ &\Phi_{1}(t,L)=H_{1}(t)+K_{1}\Phi_{3}(t,L),\label{b13}
\end{align}
where $H_{1}(t)=G_{1}(t)-\tilde{\Phi}_{1}(L),~ H_{2}(t)=G_{2}(t)-\tilde{\Phi}_{2}(0),~ H_{3}(t)=G_{3}(t)-\tilde{\Phi}_{3}(0)$.

Obviously, we have the following two facts:\\
$\mathbf{1}:~$By~\eqref{b1}, we get
\begin{align*}
\lambda_{1}(\Phi+\tilde{\Phi})<0<\lambda_{2}(\Phi+\tilde{\Phi})<\lambda_{3}(\Phi+\tilde{\Phi}),\quad \forall \Phi\in \Re,
\end{align*}
where $\Re$ is a small neighborhood of $O=(0,0,0)^{\top}$ corresponding to $\wp$.\\
$\mathbf{2}:~$Denote $ \mu_{i}(\Phi+\tilde{\Phi})=\lambda_{i}^{-1}(\Phi+\tilde{\Phi})(i=1,2,3)$, then there exists a positive constant $\mathcal{K}$, such that
\begin{align}\label{B13}
\mathop{\max}\limits_{i=1,2,3}\mathop{\sup}\limits_{\Phi\in \Re}|\mu_{i}(\Phi+\tilde{\Phi})|\leq \mathcal{K}.
\end{align}

Before stating the main results, we explain the notations for the use throughout this article.~$C$ always denotes a generic positive constant, which is different in different place. $C_{E}, C_{S}, C_{S}^{*}, C_{R}, M_{j}(j=0,1,2,3,\ldots)$ represent the specific constants. Next, we give the main results of this paper. Here we assume that there exists a small constant $\epsilon_{0}>0$, such that the inflow sound speed
\begin{align}
&c_{-}<\epsilon_{0}\label{a4}
\end{align}
holds in advance.
\begin{theorem}\label{t1}
(Existence of the time-periodic solution) For any given $M_{0}>0$, there exists a small enough constant $\epsilon_{1}\in(0,\epsilon_{0})$ and a positive constant $C_{E}$ such that for any given $0<\epsilon<\epsilon_{1}$ and any given $C^{1}$ smooth functions $H_{i}(t)(i=1,2,3)$ satisfying
\begin{align}
H_{i}(t+P)=H_{i}(t),\quad t>0,\label{b14}
\end{align}
\begin{align}
&\|H_{1}(t)\|_{C^{1}(\mathbb{R}_{+})}\leq\epsilon,~~  \|H_{3}(t)\|_{C^{1}(\mathbb{R}_{+})}\leq\epsilon,\label{b15}\\
&\|H_{2}(t)-M_{0}\|_{C^{0}(\mathbb{R}_{+})}\leq \epsilon ,~~\|\partial_{t}H_{2}(t)\|_{C^{0}(\mathbb{R}_{+})}\leq \epsilon,\label{B15}
\end{align}
there exists a $C^{1}$ smooth function $\Phi_{0}(x)=({\Phi}_{1_{0}}(x),{\Phi}_{2_{0}}(x),{\Phi}_{3_{0}}(x))^{\top}$ satisfying
\begin{align}
&\|{\Phi}_{1_{0}}(x)\|_{C^{1}([0,L])}\leq C_{E}\epsilon,~~ \|{\Phi}_{3_{0}}(x)\|_{C^{1}([0,L])}\leq C_{E}\epsilon, \label{b16}\\
&\|{\Phi}_{2_{0}}(x)\|_{C^{0}([0,L])}\leq C_{E}M_{0},~~ \|\partial_{x}{\Phi}_{2_{0}}(x)\|_{C^{0}([0,L])}\leq C_{E}\epsilon,\label{B16}
\end{align}
such that the initial-boundary value problem~\eqref{b8}-\eqref{b13} admits a $C^{1}$ time-periodic solution $\Phi=\Phi^{(P)}(t,x)=(\Phi_{1}^{(P)},\Phi_{2}^{(P)},\Phi_{3}^{(P)})^{\top}$ on $D=\{(t,x)|t\in \mathbb{R}_{+},x\in [0,L]\}$ which satisfies
\begin{align}
\Phi^{(P)}(t+P,x)&=\Phi^{(P)}(t,x),\quad \forall(t,x)\in D\label{b17}
\end{align}
and
\begin{align}
\|\Phi_{1}^{(P)}\|_{C^{1}(D)}\leq C_{E}\epsilon,~~ \|\Phi_{3}^{(P)}\|_{C^{1}(D)}\leq C_{E}\epsilon,\label{b18}
\end{align}
\begin{align}
\|\Phi_{2}^{(P)}\|_{C^{0}(D)}\leq C_{E}M_{0},~~ \max\{\|\partial_{t}\Phi_{2}^{(P)}\|_{C^{0}(D)},~~\|\partial_{x}\Phi_{2}^{(P)}\|_{C^{0}(D)}\}\leq C_{E}\epsilon.\label{B18}
\end{align}
\end{theorem}
\begin{theorem}\label{t2}
($C^{0}$ Stability of the time-periodic solution) There exists a small constant $\epsilon_{2}\in(0,\epsilon_{1})$ and a constant $C_{S}>0$, such that for any given $\epsilon\in(0,\epsilon_{2})$ and any given $C^{1}$ smooth functions $\Phi_{0}=\Phi_{0}(x)$ and $H_{i}(t)(i=1,2,3)$ satisfying \eqref{b16}-\eqref{B16} and \eqref{b14}-\eqref{B15} with certain compatibilities, the initial-boundary value problem \eqref{b8}-\eqref{b13} has a unique global $C^{1}$ classical solution $\Phi=\Phi(t,x)$ on $D=\{(t,x)|t\in\mathbb{R}_{+},x\in[0,L]\}$ satisfying
\begin{align}
\|\Phi(t,\cdot)-\Phi^{(P)}(t,\cdot)\|_{C^{0}}\leq C_{S}\epsilon\xi^{[t/T_{0}]},\quad\forall t\geq0,\label{b19}
\end{align}
where $\Phi^{(P)}$, depending on $H_{i}(t)(i=1,2,3)$, is the time-periodic solution given through \theref{t1}, $\xi\in(0,1)$ is a constant, $T_{0}=\mathop{\max}\limits_{i=1,2,3}\mathop{\sup}\limits_{\Phi\in \Re}\frac{L}{|\lambda_{i}(\Phi+\tilde{\Phi})|}$ and $[t/T_{0}]$ denotes the largest integer smaller than $t/T_{0}$.
\end{theorem}
\indent The uniqueness of the time-periodic solution is a direct consequence from~\theref{t2} by taking $t\rightarrow+\infty$.
\begin{corollary}\label{t3}
(Uniqueness of the time-periodic solution) There exists a constant $\epsilon_{3}\in(0,\epsilon_{2})$, such that for any given $\epsilon\in(0,\epsilon_{3})$ and any given $C^{1}$ smooth functions $H_{i}(t)(i=1,2,3)$ satisfying~\eqref{b14}-\eqref{B15}, the corresponding time-periodic solution $\Phi=\Phi^{(P)}(t,x)$ obtained in~\theref{t1} is unique.
\end{corollary}
\begin{theorem}\label{t4}
(Regularity of the time-periodic solution) If $\varphi(S)$ is a positive $C^{3}$ smooth function and $H_{i}(t)(i=1,2,3)$ satisfy~\eqref{b14}-\eqref{B15} and possess further $W^{2,\infty}$ regularity with
\begin{align}\label{b22}
\mathop{\max}\limits_{i=1,2,3}\|H_{i}''(t)\|_{L^{\infty}(\mathbb{R}_{+})}\leq M_{0}<+\infty,
\end{align}
then there exist constants $C_{R}>0$ and $\epsilon_{4}\in(0,\epsilon_{1})$, such that for any given $\epsilon\in(0,\epsilon_{4})$, the time-periodic solution $\Phi=\Phi^{(P)}(t,x)$ provided by~\theref{t1} is also a $W^{2,\infty}$ function with
\begin{align}\label{b23}
\max\{\|\partial_{t}^{2}\Phi^{(P)}\|_{L^{\infty}(D)},\|\partial_{t}\partial_{x}\Phi^{(P)}\|_{L^{\infty}(D)},\|\partial_{x}^{2}\Phi^{(P)}\|_{L^{\infty}(D)}\}
\leq(1+\mathcal{K})^{2}C_{R}<+\infty,
\end{align}
where $\mathcal{K}$ is defined in~\eqref{B13}.
\end{theorem}
\begin{theorem}\label{t5}
($C^{1}$ Stability of the time-periodic solution) Assume that~\eqref{b22} holds, then there exist constants $C_{S}^{*}>0$ and $\epsilon_{5}\in(0,\min\{\epsilon_{2},\epsilon_{4}\})$, such that for any given $\epsilon\in(0,\epsilon_{5})$, we have not only the $C^{0}$ convergence result~\eqref{b19} in~\theref{t2}, but also the $C^{1}$ exponential convergence as
\begin{align}
\max\{\|\partial_{t}\Phi(t,\cdot)-\partial_{t}\Phi^{(P)}(t,\cdot)\|_{C^{0}},\|\partial_{x}\Phi(t,\cdot)-\partial_{x}\Phi^{(P)}(t,\cdot)\|_{C^{0}}\}
\leq (1+\mathcal{K})&C_{S}^{*}\epsilon\xi^{[t/T_{0}]},\notag\\
&\forall t\geq0.\label{b24}
\end{align}
\end{theorem}
\begin{remark}\label{r1}
In the conditions~\eqref{B15} and~\eqref{B16}, we assume that the $C^{0}$ norm of $H_{2}(t)$ or ${\Phi}_{2_{0}}(x)$ can be large, while the homogeneous $C^{1}$ norm is suitably small. This means that the amplitude of the entropy is allowed to be large, while the oscillation of the entropy must be small.
\end{remark}
%\begin{remark}\label{r1}
%If the perturbation of the entropy is small, the condition of the small inflow sound speed is no longer necessary, which can be seen from the later proof. Here we omit the detail.
%\end{remark}
\begin{remark}\label{r2}
When $\alpha(x)\equiv0$, the equations~\eqref{a1} become the standard nonisentropic Euler equations. At that case the steady subsonic background solution is a constant state $(\underline{\rho},\underline{u},\underline{S})^{\top}$ and we only need $\underline{u}\neq0$ to obtain the similar results.
\end{remark}

\section{Existence of the Time-periodic Solution}\label{s3}
\indent\indent In this section, we give the proof of~\theref{t1} by applying a newly constructed linear iteration method.

Firstly, we multiply $\mu_{i}(\Phi+\tilde{\Phi})=\lambda_{i}^{-1}(\Phi+\tilde{\Phi})(i=1,2,3)$ on both sides of the $i$-th equation of~\eqref{b8}-\eqref{b10} for $i=1,2,3$ and exchange the positions of $t$ and $x$ to get
\begin{align}
&\partial_{x}\Phi_{1}+\mu_{1}(\Phi+\tilde{\Phi})\partial_{t}\Phi_{1}\notag\\
=&\frac{\alpha}{2}(1-\frac{(\gamma+1)(\tilde{\Phi}_{1}+\tilde{\Phi}_{3})}{(\gamma+1)\tilde{\Phi}_{1}+(3-\gamma)\tilde{\Phi}_{3}})\mu_{1}
(\tilde{\Phi})\Phi_{1}+\frac{\alpha}{2}(1-\frac{(3-\gamma)(\tilde{\Phi}_{1}+\tilde{\Phi}_{3})}{(\gamma+1)\tilde{\Phi}_{1}+(3-\gamma)
\tilde{\Phi}_{3}})\mu_{1}(\Phi+\tilde{\Phi})\Phi_{3}\notag\\
&+\frac{\alpha}{2}(1-\frac{(\gamma+1)(\tilde{\Phi}_{1}+\tilde{\Phi}_{3})}
{(\gamma+1)\tilde{\Phi}_{1}+(3-\gamma)\tilde{\Phi}_{3}})\Big(\mu_{1}(\Phi+\tilde{\Phi})-\mu_{1}(\tilde{\Phi})\Big)\Phi_{1}\notag\\
&+\frac{(\gamma-1)\varphi'(\Phi_{2}+\tilde{\Phi}_{2})}{16\gamma\varphi(\Phi_{2}+\tilde{\Phi}_{2})}(\Phi_{3}+\tilde{\Phi}_{3}-\Phi_{1}-\tilde{\Phi}_{1})^{2}
\mu_{1}(\Phi+\tilde{\Phi})\frac{\partial\Phi_{2}}{\partial x},\notag\\
=&\frac{\alpha}{2}(1-\frac{(\gamma+1)(\tilde{\Phi}_{1}+\tilde{\Phi}_{3})}{(\gamma+1)\tilde{\Phi}_{1}+(3-\gamma)\tilde{\Phi}_{3}})\mu_{1}
(\tilde{\Phi})\Phi_{1}+\frac{\alpha}{2}(1-\frac{(\gamma+1)(\tilde{\Phi}_{1}+\tilde{\Phi}_{3})}{(\gamma+1)\tilde{\Phi}_{1}+(3-\gamma)\tilde{\Phi}_{3}})\mu_{1}
(\tilde{\Phi})\Phi_{3}\notag\\
&+\frac{\alpha}{2}(1-\frac{(\gamma+1)(\tilde{\Phi}_{1}+\tilde{\Phi}_{3})}
{(\gamma+1)\tilde{\Phi}_{1}+(3-\gamma)\tilde{\Phi}_{3}})\Big(\mu_{1}(\Phi+\tilde{\Phi})-\mu_{1}(\tilde{\Phi})\Big)(\Phi_{1}+\Phi_{3})\notag\\
&+\frac{\alpha(\gamma-1)(\tilde{\Phi}_{1}+\tilde{\Phi}_{3})}{(\gamma+1)\tilde{\Phi}_{1}+(3-\gamma)\tilde{\Phi}_{3}}
\mu_{1}(\Phi+\tilde{\Phi})\Phi_{3}\notag\\
&+\frac{(\gamma-1)\varphi'(\Phi_{2}+\tilde{\Phi}_{2})}{16\gamma\varphi(\Phi_{2}+\tilde{\Phi}_{2})}(\Phi_{3}+\tilde{\Phi}_{3}-\Phi_{1}-\tilde{\Phi}_{1})^{2}
\mu_{1}(\Phi+\tilde{\Phi})\frac{\partial\Phi_{2}}{\partial x},\label{c1}
\end{align}
\begin{align}
&\partial_{x}\Phi_{2}+\mu_{2}(\Phi+\tilde{\Phi})\partial_{t}\Phi_{2}=0,\qquad\qquad\qquad\qquad\qquad\qquad\qquad\qquad\qquad\qquad
\qquad\quad
\label{c2}
\end{align}
\begin{align}
&\partial_{x}\Phi_{3}+\mu_{3}(\Phi+\tilde{\Phi})\partial_{t}\Phi_{3}\notag\\
=&\frac{\alpha}{2}(1-\frac{(\gamma+1)(\tilde{\Phi}_{1}+\tilde{\Phi}_{3})}{(3-\gamma)\tilde{\Phi}_{1}+(\gamma+1)\tilde{\Phi}_{3}})\mu_{3}
(\tilde{\Phi})\Phi_{3}+\frac{\alpha}{2}(1-\frac{(3-\gamma)(\tilde{\Phi}_{1}+\tilde{\Phi}_{3})}{(3-\gamma)\tilde{\Phi}_{1}+(\gamma+1)\tilde{\Phi}_{3}})\mu_{3}
(\Phi+\tilde{\Phi})\Phi_{1}\notag\\
&+\frac{\alpha}{2}(1-\frac{(\gamma+1)(\tilde{\Phi}_{1}+\tilde{\Phi}_{3})}
{(3-\gamma)\tilde{\Phi}_{1}+(\gamma+1)\tilde{\Phi}_{3}})\Big(\mu_{3}(\Phi+\tilde{\Phi})-\mu_{3}(\tilde{\Phi})\Big)\Phi_{3}\notag\\
&+\frac{(\gamma-1)\varphi'(\Phi_{2}+\tilde{\Phi}_{2})}{16\gamma\varphi(\Phi_{2}+\tilde{\Phi}_{2})}(\Phi_{3}+\tilde{\Phi}_{3}-\Phi_{1}-\tilde{\Phi}_{1})^{2}
\mu_{3}(\Phi+\tilde{\Phi})\frac{\partial\Phi_{2}}{\partial x}\notag\\
=&\frac{\alpha}{2}(1-\frac{(\gamma+1)(\tilde{\Phi}_{1}+\tilde{\Phi}_{3})}{(3-\gamma)\tilde{\Phi}_{1}+(\gamma+1)\tilde{\Phi}_{3}})\mu_{3}
(\tilde{\Phi})\Phi_{3}+\frac{\alpha}{2}(1-\frac{(\gamma+1)(\tilde{\Phi}_{1}+\tilde{\Phi}_{3})}{(3-\gamma)\tilde{\Phi}_{1}+(\gamma+1)\tilde{\Phi}_{3}})\mu_{3}
(\tilde{\Phi})\Phi_{1}\notag\\
&+\frac{\alpha}{2}(1-\frac{(\gamma+1)(\tilde{\Phi}_{1}+\tilde{\Phi}_{3})}
{(3-\gamma)\tilde{\Phi}_{1}+(\gamma+1)\tilde{\Phi}_{3}})\Big(\mu_{3}(\Phi+\tilde{\Phi})-\mu_{3}(\tilde{\Phi})\Big)(\Phi_{1}+\Phi_{3})\notag\\
&+\frac{\alpha(\gamma-1)(\tilde{\Phi}_{1}+\tilde{\Phi}_{3})}{(3-\gamma)\tilde{\Phi}_{1}+(\gamma+1)\tilde{\Phi}_{3}}
\mu_{3}(\Phi+\tilde{\Phi})\Phi_{1}\notag\\
&+\frac{(\gamma-1)\varphi'(\Phi_{2}+\tilde{\Phi}_{2})}{16\gamma\varphi(\Phi_{2}+\tilde{\Phi}_{2})}(\Phi_{3}+\tilde{\Phi}_{3}-\Phi_{1}-\tilde{\Phi}_{1})^{2}
\mu_{3}(\Phi+\tilde{\Phi})\frac{\partial\Phi_{2}}{\partial x}.\label{c3}
\end{align}
In the following proof, we only give the proof when $|K_{1}|<1, |K_{3}|<1$. If $|K_{1}|=1$ or $|K_{3}|=1$ occurs, we only need to make some transformation to turn it into the above form. Moreover, for convenience, we assume $|K_{2}|\leq1$. The proof process is similar when $|K_{2}|>1$.

Then, we establish the following "initial"-value problem of linearized system
\begin{align}
&\partial_{x}\Phi_{1}^{(l)}+\mu_{1}(\Phi^{(l-1)}+\tilde{\Phi})\partial_{t}\Phi_{1}^{(l)}\notag\\
=&\frac{\alpha}{2}(1-\frac{(\gamma+1)(\tilde{\Phi}_{1}+\tilde{\Phi}_{3})}{(\gamma+1)\tilde{\Phi}_{1}+(3-\gamma)\tilde{\Phi}_{3}})\mu_{1}
(\tilde{\Phi})\Phi_{1}^{(l)}+\frac{\alpha}{2}(1-\frac{(\gamma+1)(\tilde{\Phi}_{1}+\tilde{\Phi}_{3})}{(\gamma+1)\tilde{\Phi}_{1}
+(3-\gamma)\tilde{\Phi}_{3}})\mu_{1}(\tilde{\Phi})\Phi_{3}^{(l-1)}\notag\\
&+\frac{\alpha}{2}(1-\frac{(\gamma+1)(\tilde{\Phi}_{1}+\tilde{\Phi}_{3})}
{(\gamma+1)\tilde{\Phi}_{1}+(3-\gamma)\tilde{\Phi}_{3}})\Big(\mu_{1}(\Phi^{(l-1)}+\tilde{\Phi})-\mu_{1}(\tilde{\Phi})\Big)(\Phi_{1}^{(l-1)}
+\Phi_{3}^{(l-1)})\notag\\
&+\frac{\alpha(\gamma-1)(\tilde{\Phi}_{1}+\tilde{\Phi}_{3})}{(\gamma+1)\tilde{\Phi}_{1}+(3-\gamma)\tilde{\Phi}_{3}}
\mu_{1}(\Phi^{(l-1)}+\tilde{\Phi})\Phi_{3}^{(l-1)}\notag\\
&+\frac{(\gamma-1)\varphi'(\Phi^{(l-1)}_{2}+\tilde{\Phi}_{2})}{16\gamma\varphi(\Phi^{(l-1)}_{2}+\tilde{\Phi}_{2})}(\Phi_{3}^{(l-1)}+\tilde{\Phi}_{3}
-\Phi_{1}^{(l-1)}-\tilde{\Phi}_{1})^{2}\mu_{1}(\Phi^{(l-1)}+\tilde{\Phi})\frac{\partial\Phi_{2}^{(l-1)}}{\partial x},\label{c4}\\
&x=L:~~ \Phi_{1}^{(l)}(t,L)=\mathcal{H}_{1}(t)+K_{1}\Phi_{3}^{(l-1)}(t,L),\label{c5}\\
\notag\\
&\partial_{x}\Phi_{2}^{(l)}+\mu_{2}(\Phi^{(l-1)}+\tilde{\Phi})\partial_{t}\Phi_{2}^{(l)}=0,\label{c6}\\
&x=0:~~ \Phi_{2}^{(l)}(t,0)=\mathcal{H}_{2}(t)+K_{2}\Phi_{1}^{(l)}(t,0),\label{c7}
\end{align}
\begin{align}
&\partial_{x}\Phi_{3}^{(l)}+\mu_{3}(\Phi^{(l-1)}+\tilde{\Phi})\partial_{t}\Phi_{3}^{(l)}\notag\\
=&\frac{\alpha}{2}(1-\frac{(\gamma+1)(\tilde{\Phi}_{1}+\tilde{\Phi}_{3})}{(3-\gamma)\tilde{\Phi}_{1}+(\gamma+1)\tilde{\Phi}_{3}})\mu_{3}
(\tilde{\Phi})\Phi_{3}^{(l)}+\frac{\alpha}{2}(1-\frac{(\gamma+1)(\tilde{\Phi}_{1}+\tilde{\Phi}_{3})}{(3-\gamma)\tilde{\Phi}_{1}+(\gamma+1)\tilde{\Phi}_{3}})\mu_{3}
(\tilde{\Phi})\Phi_{1}^{(l-1)}\notag\\
&+\frac{\alpha}{2}(1-\frac{(\gamma+1)(\tilde{\Phi}_{1}+\tilde{\Phi}_{3})}
{(3-\gamma)\tilde{\Phi}_{1}+(\gamma+1)\tilde{\Phi}_{3}})\Big(\mu_{3}(\Phi^{(l-1)}+\tilde{\Phi})-\mu_{3}(\tilde{\Phi})\Big)(\Phi_{1}^{(l-1)}
+\Phi_{3}^{(l-1)})\notag\\
&+\frac{\alpha(\gamma-1)(\tilde{\Phi}_{1}+\tilde{\Phi}_{3})}{(3-\gamma)\tilde{\Phi}_{1}+(\gamma+1)\tilde{\Phi}_{3}}
\mu_{3}(\Phi^{(l-1)}+\tilde{\Phi})\Phi_{1}^{(l-1)}\notag\\
&+\frac{(\gamma-1)\varphi'(\Phi^{(l-1)}_{2}+\tilde{\Phi}_{2})}{16\gamma\varphi(\Phi^{(l-1)}_{2}+\tilde{\Phi}_{2})}(\Phi_{3}^{(l-1)}+\tilde{\Phi}_{3}
-\Phi_{1}^{(l-1)}-\tilde{\Phi}_{1})^{2}\mu_{3}(\Phi^{(l-1)}+\tilde{\Phi})\frac{\partial\Phi_{2}^{(l-1)}}{\partial x}.\label{c8}\\
&x=0:~~ \Phi_{3}^{(l)}(t,0)=\mathcal{H}_{3}(t)+K_{3}\Phi_{1}^{(l-1)}(t,0),\label{c9}
\end{align}
where
\begin{align*}
\mathcal{H}_{i}(t)=\left\{
\begin{aligned}
&H_{i}(t),\quad t\geq0,\\
&\tilde{H}_{i}(t),\quad t<0,
\end{aligned}\right.
\end{align*}
are the time-periodic extensions of $H_{i}(t)(i=1,2,3)$.

Next, we start the iteration~\eqref{c4}-\eqref{c9} from
\begin{align}
\Phi_{1}^{(0)}(t,x)=0,~\Phi_{2}^{(0)}(t,x)=M_{0},~\Phi_{3}^{(0)}(t,x)=0. \label{c10}
\end{align}

\begin{proposition}\label{p1}
There exists a small enough constant $\epsilon_{1}>0$, positive constants $M_{1}, M_{2}>0$ and $\eta\in(0,1)$, such that for any given $\epsilon\in(0,\epsilon_{1})$ and $l\in \mathbb{N}_{+}$, the following estimates hold
\begin{align}
&\Phi^{(l)}(t+P,x)=\Phi^{(l)}(t,x),\quad\forall(t,x)\in \mathbb{R}\times[0,L],\label{c11}
\end{align}
\begin{align}
&\|\Phi_{1}^{(l)}\|_{C^{0}(D)}\leq M_{1}\epsilon,\quad \|\partial_{t}\Phi_{1}^{(l)}\|_{C^{0}(D)}\leq M_{1}\epsilon,\quad \|\partial_{x}\Phi_{1}^{(l)}\|_{C^{0}(D)}\leq M_{2}\epsilon, \label{c12}\\
&\|\Phi_{2}^{(l)}\|_{C^{0}(D)}\leq 2M_{0},\quad \|\partial_{t}\Phi_{2}^{(l)}\|_{C^{0}(D)}\leq M_{1}\epsilon,\quad \|\partial_{x}\Phi_{2}^{(l)}\|_{C^{0}(D)}\leq M_{2}\epsilon,\label{c13}\\
&\|\Phi_{3}^{(l)}\|_{C^{0}(D)}\leq M_{1}\epsilon,\quad \|\partial_{t}\Phi_{3}^{(l)}\|_{C^{0}(D)}\leq M_{1}\epsilon,\quad \|\partial_{x}\Phi_{3}^{(l)}\|_{C^{0}(D)}\leq M_{2}\epsilon,\label{c14}
\end{align}
\begin{align}
\|\Phi^{(l)}-\Phi^{(l-1)}\|_{C^{0}(D)}\leq M_{1}\epsilon\eta^{l}, \label{c15}
\end{align}
\begin{align}
\mathop{\max}\limits_{i=1,2,3}\{\varpi(\delta|\partial_{t}\Phi_{i}^{(l)})+\varpi(\delta|\partial_{x}\Phi_{i}^{(l)})\}
\leq(\frac{1}{3}+\frac{1}{2}\mathcal{K})\Omega(\delta), \label{c16}
\end{align}
where
\begin{align}
\varpi(\delta|h)=\mathop{\sup}\limits_{\substack{|t_{1}-t_{2}|\leq\delta\\|x_{1}-x_{2}|\leq\delta}}|h(t_{1},x_{1})-h(t_{2},x_{2})|\notag
\end{align}
and $\Omega(\delta)$ is a continuous function of $\delta\in(0,1)$ which is independent of $l$ and satisfies
$$
\mathop{\lim}\limits_{\delta\rightarrow0^{+}}\Omega(\delta)=0.
$$
\end{proposition}
\begin{proof}
We prove the a priori estimates~\eqref{c11}-\eqref{c16} inductively, i.e., for each $l\in \mathbb{N}_{+}$, we show
\begin{align}
&\Phi_{i}^{(l)}(t+P,x)=\Phi_{i}^{(l)}(t,x), \quad\forall(t,x)\in \mathbb{R}\times[0,L], \forall i=1,2,3,\label{c17}\\
&\|\Phi_{1}^{(l)}\|_{C^{0}}\leq M_{1}\epsilon,~~\|\partial_{t}\Phi_{1}^{(l)}\|_{C^{0}}\leq M_{1}\epsilon,~~
\|\partial_{x}\Phi_{1}^{(l)}\|_{C^{0}}\leq M_{2}\epsilon,\label{c18}\\
&\|\Phi_{2}^{(l)}\|_{C^{0}}\leq 2M_{0},~~\|\partial_{t}\Phi_{2}^{(l)}\|_{C^{0}}\leq M_{1}\epsilon,~~\|\partial_{x}\Phi_{2}^{(l)}\|_{C^{0}}\leq M_{2}\epsilon,\label{c19}\\
&\|\Phi_{3}^{(l)}\|_{C^{0}}\leq M_{1}\epsilon,~~\|\partial_{t}\Phi_{3}^{(l)}\|_{C^{0}}\leq M_{1}\epsilon,~~
\|\partial_{x}\Phi_{3}^{(l)}\|_{C^{0}}\leq M_{2}\epsilon,\label{c20}\\
&\|\Phi_{i}^{(l)}-\Phi_{i}^{(l-1)}\|_{C^{0}}\leq M_{1}\epsilon\eta^{l},\quad\quad\quad\quad\quad\quad\quad\quad~~ i=1,2,3,\label{c21}\\
&\varpi(\delta|\partial_{t}\Phi_{i}^{(l)}(\cdot,x))\leq\frac{1}{8[\mathcal{K}+1]}\Omega(\delta),\quad i=1,2,3,~\forall x\in[0,L]
\label{c23}
\end{align}
and
\begin{align}
&\max\{\varpi(\delta|\partial_{t}\Phi_{i}^{(l)})+\varpi(\delta|\partial_{x}\Phi_{i}^{(l)})\}
\leq(\frac{1}{3}+\frac{1}{2}\mathcal{K})\Omega(\delta), \quad i=1,2,3\label{c25}
\end{align}
under the following hypotheses
\begin{align}
&\Phi_{i}^{(l-1)}(t+P,x)=\Phi_{i}^{(l-1)}(t,x), \quad\forall(t,x)\in \mathbb{R}\times[0,L], \forall i=1,2,3,\label{c28}\\
&\|\Phi_{1}^{(l-1)}\|_{C^{0}}\leq M_{1}\epsilon,~~\|\partial_{t}\Phi_{1}^{(l-1)}\|_{C^{0}}\leq M_{1}\epsilon,~~
\|\partial_{x}\Phi_{1}^{(l-1)}\|_{C^{0}}\leq M_{2}\epsilon,\label{c29}\\
&\|\Phi_{2}^{(l-1)}\|_{C^{0}}\leq 2M_{0},~~\|\partial_{t}\Phi_{2}^{(l-1)}\|_{C^{0}}\leq M_{1}\epsilon,~~\|\partial_{x}\Phi_{2}^{(l-1)}\|_{C^{0}}\leq M_{2}\epsilon,\label{c30}\\
&\|\Phi_{3}^{(l-1)}\|_{C^{0}}\leq M_{1}\epsilon,~~\|\partial_{t}\Phi_{3}^{(l-1)}\|_{C^{0}}\leq M_{1}\epsilon,~~
\|\partial_{x}\Phi_{3}^{(l-1)}\|_{C^{0}}\leq M_{2}\epsilon,\label{c31}\\
&\|\Phi_{i}^{(l-1)}-\Phi_{i}^{(l-2)}\|_{C^{0}}\leq M_{1}\epsilon\eta^{l-1},\quad\quad\quad\quad i=1,2,3,~\forall l\geq2,\label{c32}\\
&\varpi(\delta|\partial_{t}\Phi_{i}^{(l-1)}(\cdot,x))\leq\frac{1}{8[\mathcal{K}+1]}\Omega(\delta), \quad i=1,2,3,~\forall x\in[0,L]\label{c34}
\end{align}
and
\begin{align}
&\max\{\varpi(\delta|\partial_{t}\Phi_{i}^{(l-1)})+\varpi(\delta|\partial_{x}\Phi_{i}^{(l-1)})\}
\leq(\frac{1}{3}+\frac{1}{2}\mathcal{K})\Omega(\delta),\label{c36}
\end{align}
where $[\mathcal{K}+1]$ represents the integer part of $\mathcal{K}+1$ and
$$\varpi(\delta|h(\cdot,x))=\mathop{\max}\limits_{|t_{1}-t_{2}|\leq\delta}|h(t_{1},x)-h(t_{2},x)|.$$

To see this, we denote
$$ F_{1}(x)=e^{\int_{x}^{L}\frac{\alpha(s)}{2}(1-\frac{(\gamma+1)(\tilde{\Phi}_{1}+\tilde{\Phi}_{3})}{(\gamma+1)\tilde{\Phi}_{1}
+(3-\gamma)\tilde{\Phi}_{3}})\mu_{1}(\tilde{\Phi}(s))ds},
~~F_{3}(x)=e^{-\int_{0}^{x}\frac{\alpha(s)}{2}(1-\frac{(\gamma+1)(\tilde{\Phi}_{1}+\tilde{\Phi}_{3})}{(3-\gamma)\tilde{\Phi}_{1}
+(\gamma+1)\tilde{\Phi}_{3}})\mu_{3}(\tilde{\Phi}(s))ds}.$$
By $\lambda_{1}(\tilde{\Phi})=\frac{(\gamma+1)\tilde{\Phi}_{1}+(3-\gamma)\tilde{\Phi}_{3}}{4}<0$ and $\tilde{u}>0$, we get
\begin{align*}
\frac{(\gamma+1)(\tilde{\Phi}_{1}+\tilde{\Phi}_{3})}{(\gamma+1)\tilde{\Phi}_{1}+(3-\gamma)\tilde{\Phi}_{3}}
=\frac{\gamma+1}{4\lambda_{1}(\tilde{\Phi})}(\tilde{\Phi}_{1}+\tilde{\Phi}_{3})
=\frac{\gamma+1}{2\lambda_{1}(\tilde{\Phi})}\tilde{u}<0,
\end{align*}
then
\begin{align}
F_{1}(x)\geq1,\label{c39}
\end{align}
\begin{align}
\frac{d}{dx}F_{1}(x)=-\frac{\alpha(x)}{2}(1-\frac{(\gamma+1)(\tilde{\Phi}_{1}+\tilde{\Phi}_{3})}{(\gamma+1)\tilde{\Phi}_{1}
+(3-\gamma)\tilde{\Phi}_{3}})\mu_{1}(\tilde{\Phi})F_{1}(x)\leq0.\label{c40}
\end{align}
Using $\lambda_{3}(\tilde{\Phi})=\frac{(3-\gamma)\tilde{\Phi}_{1}+(\gamma+1)\tilde{\Phi}_{3}}{4}>0, 1<\gamma<3$ and $\tilde{u}<\tilde{c}$, we have
\begin{align*}
\frac{(\gamma+1)(\tilde{\Phi}_{1}+\tilde{\Phi}_{3})}{(3-\gamma)\tilde{\Phi}_{1}+(\gamma+1)\tilde{\Phi}_{3}}
=\frac{\gamma+1}{4\lambda_{3}(\tilde{\Phi})}(\tilde{\Phi}_{1}+\tilde{\Phi}_{3})
=\frac{(\gamma+1)\tilde{u}}{2(\tilde{u}+\tilde{c})}
<\frac{(\gamma+1)\tilde{u}}{4\tilde{u}}<1,
\end{align*}
then
\begin{align}
&F_{3}(x)\geq1,\label{CC39}\\
&\frac{d}{dx}F_{3}(x)=-\frac{\alpha(x)}{2}(1-\frac{(\gamma+1)(\tilde{\Phi}_{1}+\tilde{\Phi}_{3})}{(3-\gamma)\tilde{\Phi}_{1}
+(\gamma+1)\tilde{\Phi}_{3}})\mu_{3}(\tilde{\Phi})F_{3}(x)\geq0.\label{c41}
\end{align}
Moreover, one has
\begin{align}
&1\leq F_{1}(x)\leq e^{-\frac{\alpha_{*}}{2}\mathcal{K}L(1+\frac{\gamma+1}{2}\mathcal{K}c_{-})}\mathop{=}\limits^{def.}\mathcal{M},\label{c42}\\
&1\leq F_{3}(x)\leq \mathcal{M}.\label{c43}
\end{align}

Now, we turn problem~\eqref{c4}-\eqref{c5} and~\eqref{c8}-\eqref{c9} into the system of $F_{1}(x)\Phi_{1}^{(l)}$ and $F_{3}(x)\Phi_{3}^{(l)}$ as follows
\begin{align}
&\partial_{x}(F_{1}\Phi_{1}^{(l)})+\mu_{1}(\Phi^{(l-1)}+\tilde{\Phi})\partial_{t}(F_{1}\Phi_{1}^{(l)})\notag\\
=&\frac{\alpha}{2}(1-\frac{(\gamma+1)(\tilde{\Phi}_{1}+\tilde{\Phi}_{3})}
{(\gamma+1)\tilde{\Phi}_{1}+(3-\gamma)\tilde{\Phi}_{3}})F_{1}\Big(\mu_{1}(\Phi^{(l-1)}+\tilde{\Phi})-\mu_{1}(\tilde{\Phi})\Big)
(\Phi_{1}^{(l-1)}+\Phi_{3}^{(l-1)})\notag\\
&+\frac{(\gamma-1)\varphi'(\Phi_{2}^{(l-1)}+\tilde{\Phi}_{2})}{16\gamma\varphi(\Phi_{2}^{(l-1)}+\tilde{\Phi}_{2})}
(\Phi_{3}^{(l-1)}+\tilde{\Phi}_{3}
-\Phi_{1}^{(l-1)}-\tilde{\Phi}_{1})^{2}F_{1}\mu_{1}(\Phi^{(l-1)}+\tilde{\Phi})\frac{\partial\Phi_{2}^{(l-1)}}{\partial x}\notag\\
&+\frac{\alpha}{2}(1-\frac{(\gamma+1)(\tilde{\Phi}_{1}+\tilde{\Phi}_{3})}
{(\gamma+1)\tilde{\Phi}_{1}+(3-\gamma)\tilde{\Phi}_{3}})F_{1}\mu_{1}(\tilde{\Phi})\Phi_{3}^{(l-1)}\notag\\
&+\frac{\alpha(\gamma-1)(\tilde{\Phi}_{1}+\tilde{\Phi}_{3})}
{(\gamma+1)\tilde{\Phi}_{1}+(3-\gamma)\tilde{\Phi}_{3}}F_{1}\mu_{1}(\Phi^{(l-1)}+\tilde{\Phi})\Phi_{3}^{(l-1)},\label{c44}\\
\notag\\
&x=L:~~~~F_{1}(L)\Phi_{1}^{(l)}(t,L)=\Phi_{1}^{(l)}(t,L)=\mathcal{H}_{1}(t)+K_{1}\Phi_{3}^{(l-1)}(t,L),\label{c45}
\end{align}
\begin{align}
&\partial_{x}(F_{3}\Phi_{3}^{(l)})+\mu_{3}(\Phi^{(l-1)}+\tilde{\Phi})\partial_{t}(F_{3}\Phi_{3}^{(l)})\notag\\
=&\frac{\alpha}{2}(1-\frac{(\gamma+1)(\tilde{\Phi}_{1}+\tilde{\Phi}_{3})}
{(3-\gamma)\tilde{\Phi}_{1}+(\gamma+1)\tilde{\Phi}_{3}})F_{3}\Big(\mu_{3}(\Phi^{(l-1)}+\tilde{\Phi})-\mu_{3}(\tilde{\Phi})\Big)
(\Phi_{1}^{(l-1)}+\Phi_{3}^{(l-1)})\notag\\
&+\frac{(\gamma-1)\varphi'(\Phi_{2}^{(l-1)}+\tilde{\Phi}_{2})}{16\gamma\varphi(\Phi_{2}^{(l-1)}+\tilde{\Phi}_{2})}(\Phi_{3}^{(l-1)}+\tilde{\Phi}_{3}
-\Phi_{1}^{(l-1)}-\tilde{\Phi}_{1})^{2}F_{3}\mu_{3}(\Phi^{(l-1)}+\tilde{\Phi})\frac{\partial\Phi_{2}^{(l-1)}}{\partial x}\notag\\
&+\frac{\alpha}{2}(1-\frac{(\gamma+1)(\tilde{\Phi}_{1}+\tilde{\Phi}_{3})}
{(3-\gamma)\tilde{\Phi}_{1}+(\gamma+1)\tilde{\Phi}_{3}})F_{3}\mu_{3}(\tilde{\Phi})\Phi_{1}^{(l-1)}\notag\\
&+\frac{\alpha(\gamma-1)(\tilde{\Phi}_{1}+\tilde{\Phi}_{3})}
{(3-\gamma)\tilde{\Phi}_{1}+(\gamma+1)\tilde{\Phi}_{3}}F_{3}\mu_{3}(\Phi^{(l-1)}+\tilde{\Phi})\Phi_{1}^{(l-1)},\label{c46}\\
\notag\\
&x=0:~~~~F_{3}(0)\Phi_{3}^{(l)}(t,0)=\Phi_{3}^{(l)}(t,0)=\mathcal{H}_{3}(t)+K_{3}\Phi_{1}^{(l-1)}(t,0).\label{c47}
\end{align}
By~\eqref{b14} and~\eqref{c28}, we get that if $F_{1}(x)\Phi_{1}^{(l)}(t,x)$, $\Phi_{2}^{(l)}(t,x)$ and $F_{3}(x)\Phi_{3}^{(l)}(t,x)$ solve problem~\eqref{c44}-\eqref{c45}, \eqref{c6}-\eqref{c7} and~\eqref{c46}-\eqref{c47} respectively, so does $F_{1}(x)\Phi_{1}^{(l)}(t+P,x)$, $\Phi_{2}^{(l)}(t+P,x)$ and $F_{3}(x)\Phi_{3}^{(l)}(t+P,x)$.  Then we get $F_{1}(x)\Phi_{1}^{(l)}(t+P,x)=F_{1}(x)\Phi_{1}^{(l)}(t,x)$, $\Phi_{2}^{(l)}(t+P,x)=\Phi_{2}^{(l)}(t,x)$ and $F_{3}(x)\Phi_{3}^{(l)}(t+P,x)=F_{3}(x)\Phi_{3}^{(l)}(t,x)$ by the uniqueness of this linear system. Thus~\eqref{c17} is proved.

Next, we start to show the estimates~\eqref{c18}-\eqref{c25}. Since the proof of all estimates of $\Phi_{3}^{(l)}$ is similar to that of $\Phi_{1}^{(l)}$, we only prove the estimates of $\Phi_{1}^{(l)}$ and $\Phi_{2}^{(l)}$ here.

Define the characteristic curve $t=t_{i}^{(l)}(x;t_{0},x_{0})(i=1,2,3)$ as the following form:
\begin{align}
\left\{
\begin{aligned}
&\frac{dt_{i}^{(l)}}{dx}(x;t_{0},x_{0})=\mu_{i}(\Phi^{(l-1)}(t_{i}^{(l)}(x;t_{0},x_{0}),x)+\tilde{\Phi}),\\
&t_{i}^{(l)}(x_{0};t_{0},x_{0})=t_{0}.
\end{aligned}\right.\label{c48}
\end{align}
Let $M_{1}=\frac{100\mathcal{M}}{1-K}$ with $K=\max\{|K_{1}|,|K_{3}|\}$, then
\begin{align}
M_{1}\geq|K_{1}|M_{1}+100\mathcal{M},~~M_{1}\geq|K_{3}|M_{1}+100\mathcal{M}.\label{c49}
\end{align}
Then at the boundary $x=L$, by~\eqref{b15} and~\eqref{c31}, we have
\begin{align}
\|\Phi_{1}^{(l)}(\cdot,L)\|_{C^{0}(R)}\leq\epsilon+|K_{1}|M_{1}\epsilon\leq M_{1}\epsilon-99\mathcal{M}\epsilon.\label{c50}
\end{align}
Integrating~\eqref{c44} along the $1$-st characteristic curve $t=t_{1}^{(l)}(y;t,x)$ from $L$ to $x$, we get
\begin{align}
&F_{1}(x)\Phi_{1}^{(l)}(t,x)-F_{1}(L)\Phi_{1}^{(l)}(t_{1}^{(l)}(L;t,x),L)\notag\\
=&\int_{L}^{x}\Big(\frac{\alpha}{2}(1-\frac{(\gamma+1)(\tilde{\Phi}_{1}+\tilde{\Phi}_{3})}
{(\gamma+1)\tilde{\Phi}_{1}+(3-\gamma)\tilde{\Phi}_{3}})F_{1}\Big(\mu_{1}(\Phi^{(l-1)}+\tilde{\Phi})-\mu_{1}(\tilde{\Phi})\Big)
(\Phi_{1}^{(l-1)}+\Phi_{3}^{(l-1)})\notag\\
&+\frac{(\gamma-1)\varphi'(\Phi_{2}^{(l-1)}+\tilde{\Phi}_{2})}{16\gamma\varphi(\Phi_{2}^{(l-1)}+\tilde{\Phi}_{2})}(\Phi_{3}^{(l-1)}
+\tilde{\Phi}_{3}
-\Phi_{1}^{(l-1)}-\tilde{\Phi}_{1})^{2}F_{1}\mu_{1}(\Phi^{(l-1)}+\tilde{\Phi})\frac{\partial\Phi_{2}^{(l-1)}}{\partial x}\notag\\
&+\frac{\alpha(\gamma-1)(\tilde{\Phi}_{1}+\tilde{\Phi}_{3})}
{(\gamma+1)\tilde{\Phi}_{1}+(3-\gamma)\tilde{\Phi}_{3}}F_{1}\mu_{1}(\Phi^{(l-1)}+\tilde{\Phi})\Phi_{3}^{(l-1)}\Big)
\big|_{\tau=t_{1}^{(l)}(y;t,x)}dy\notag\\
&+\int_{L}^{x}(-\frac{d}{dy}F_{1}(y))\Phi_{3}^{(l-1)}(t_{1}^{(l)}(y;t,x),y)dy,\label{c51}
\end{align}
where $F_{1}(L)=1$.

It is worthy to notice that by~\eqref{a2}, \eqref{A3} and~\eqref{B13}-\eqref{a4}, we have
\begin{align}
&|\frac{\alpha(\gamma-1)(\tilde{\Phi}_{1}+\tilde{\Phi}_{3})}
{(\gamma+1)\tilde{\Phi}_{1}+(3-\gamma)\tilde{\Phi}_{3}}\mu_{1}(\Phi^{(l-1)}+\tilde{\Phi})|\notag\\
=&|\frac{\alpha}{2}(\gamma-1)\mu_{1}(\tilde{\Phi})\mu_{1}(\Phi^{(l-1)}+\tilde{\Phi})\tilde{u}|\notag\\
\leq&-\alpha_{*}(\gamma-1)\mathcal{K}^{2}c_{-}<C\epsilon_{0} \label{c52}
\end{align}
and using~\eqref{B13}-\eqref{a4}, \eqref{c29} and~\eqref{c31}, on has
\begin{align}
&\big|\frac{(\gamma-1)\varphi'(\Phi_{2}^{(l-1)}+\tilde{\Phi}_{2})}{16\gamma\varphi(\Phi_{2}^{(l-1)}+\tilde{\Phi}_{2})}(\Phi_{3}^{(l-1)}
+\tilde{\Phi}_{3}
-\Phi_{1}^{(l-1)}-\tilde{\Phi}_{1})^{2}\mu_{1}(\Phi^{(l-1)}+\tilde{\Phi})\big|\notag\\
=&\big|\Big(\frac{(\gamma-1)\varphi'(\Phi_{2}^{(l-1)}+\tilde{\Phi}_{2})}{16\gamma\varphi(\Phi_{2}^{(l-1)}+\tilde{\Phi}_{2})}(\Phi_{3}^{(l-1)}
-\Phi_{1}^{(l-1)})^{2}
+\frac{(\gamma-1)\varphi'(\Phi_{2}^{(l-1)}+\tilde{\Phi}_{2})}{8\gamma\varphi(\Phi_{2}^{(l-1)}+\tilde{\Phi}_{2})}(\Phi_{3}^{(l-1)}
-\Phi_{1}^{(l-1)})(\tilde{\Phi}_{3}-\tilde{\Phi}_{1})\notag\\
&+\frac{(\gamma-1)\varphi'(\Phi_{2}^{(l-1)}+\tilde{\Phi}_{2})}{16\gamma\varphi(\Phi_{2}^{(l-1)}+\tilde{\Phi}_{2})}(\tilde{\Phi}_{3}
-\tilde{\Phi}_{1})^{2}\Big)
\mu_{1}(\Phi^{(l-1)}+\tilde{\Phi})\big|\notag\\
\leq&\Big(\big|\frac{(\gamma-1)\varphi'(\Phi_{2}^{(l-1)}+\tilde{\Phi}_{2})}{4\gamma\varphi(\Phi_{2}^{(l-1)}+\tilde{\Phi}_{2})}\big|
(M_{1}\epsilon)^{2}
+\big|\frac{\varphi'(\Phi_{2}^{(l-1)}+\tilde{\Phi}_{2})}{\gamma\varphi(\Phi_{2}^{(l-1)}+\tilde{\Phi}_{2})}\big|M_{1}\epsilon\epsilon_{0}
+\big|\frac{\varphi'(\Phi_{2}^{(l-1)}+\tilde{\Phi}_{2})}{\gamma(\gamma-1)\varphi(\Phi_{2}^{(l-1)}+\tilde{\Phi}_{2})}\big|\epsilon_{0}^{2}\Big)
\mathcal{K}\notag\\
\leq&C\epsilon_{0}^{2}.\label{c53}
\end{align}
Then by~\eqref{c29}-\eqref{c31}, \eqref{c42} and~\eqref{c50}-\eqref{c53}, we get
\begin{align}
\|\Phi_{1}^{(l)}\|_{C^{0}}\leq&\frac{M_{1}\epsilon-99\mathcal{M}\epsilon}{F_{1}(x)}+C\epsilon^{2}
+\frac{F_{1}(x)-1}{F_{1}(x)}M_{1}\epsilon+C\epsilon_{0}\epsilon\notag\\
=&M_{1}\epsilon-\frac{99\mathcal{M}\epsilon}{F_{1}(x)}+C\epsilon^{2}+C\epsilon_{0}\epsilon\notag\\
\leq&\hbar_{1}M_{1}\epsilon,\label{c54}
\end{align}
where the constant $\hbar_{1}$ satisfies $0<\hbar_{1}<\eta<1$.

At the boundary $x=0$, by~\eqref{B15} and~\eqref{c54}, we have
\begin{align}
\|\Phi_{2}^{(l)}(\cdot,0)\|_{C^{0}(\mathbb{R})}\leq M_{0}+\epsilon+|K_{2}|\hbar_{1}M_{1}\epsilon,\label{c55}
\end{align}
then with the aid of~\eqref{c6} and~\eqref{c55}, we get
\begin{align}
\|\Phi_{2}^{(l)}\|_{C^{0}}=\|\Phi_{2}^{(l)}(\cdot,0)\|_{C^{0}(\mathbb{R})}\leq M_{0}+\epsilon+|K_{2}|\hbar_{1}M_{1}\epsilon< 2M_{0}.\label{c56}
\end{align}

Next, we give the $C^{1}$ estimates of $\Phi_{i}^{(l)}(i=1,2,3)$.
Letting
$$\Psi_{i}^{(l)}=\partial_{t}\Phi_{i}^{(l)},\quad i=1,2,3,~~\forall l\in\mathbb{N},$$
differentiating equations~\eqref{c4}, \eqref{c6} and the boundary conditions~\eqref{c5}, \eqref{c7} with respect to $t$, we get
\begin{align}
&\partial_{x}\Psi_{1}^{(l)}+\mu_{1}(\Phi^{(l-1)}+\tilde{\Phi})\partial_{t}\Psi_{1}^{(l)}\notag\\
=&-\frac{\varphi'(\Phi_{2}^{(l-1)}+\tilde{\Phi}_{2})}{4\gamma\varphi(\Phi_{2}^{(l-1)}+\tilde{\Phi}_{2})}(\Phi_{3}^{(l-1)}+\tilde{\Phi}_{3}
-\Phi_{1}^{(l-1)}-\tilde{\Phi}_{1})\Big(\frac{\partial\Psi_{2}^{(l-1)}}{\partial x}+\mu_{1}(\Phi^{(l-1)}+\tilde{\Phi})\frac{\partial\Psi_{2}^{(l-1)}}{\partial t}\Big)\notag\\
&+\frac{\alpha}{2}(1-\frac{(\gamma+1)(\tilde{\Phi}_{1}+\tilde{\Phi}_{3})}{(\gamma+1)\tilde{\Phi}_{1}+(3-\gamma)\tilde{\Phi}_{3}})
\mu_{1}(\tilde{\Phi})\Psi_{3}^{(l-1)}
+\frac{\alpha}{2}(1-\frac{(\gamma+1)(\tilde{\Phi}_{1}+\tilde{\Phi}_{3})}{(\gamma+1)\tilde{\Phi}_{1}+(3-\gamma)\tilde{\Phi}_{3}})
\mu_{1}(\tilde{\Phi})\Psi_{1}^{(l)}\notag\\
&-\Big(\frac{\partial\mu_{1}(\Phi^{(l-1)}+\tilde{\Phi})}{\partial\Phi_{1}^{(l-1)}}\Psi_{1}^{(l-1)}
+\frac{\partial\mu_{1}(\Phi^{(l-1)}+\tilde{\Phi})}{\partial\Phi_{3}^{(l-1)}}\Psi_{3}^{(l-1)}\Big)\Psi_{1}^{(l)}\notag\\
&+\frac{\alpha(\gamma-1)(\tilde{\Phi}_{1}+\tilde{\Phi}_{3})}{(\gamma+1)\tilde{\Phi}_{1}+(3-\gamma)\tilde{\Phi}}_{3}\mu_{1}(\Phi^{(l-1)}
+\tilde{\Phi})\Psi_{3}^{(l-1)}\notag\\
&+\frac{(\gamma-1)\Big(\varphi(\Phi_{2}^{(l-1)}+\tilde{\Phi}_{2})\varphi''(\Phi_{2}^{(l-1)}+\tilde{\Phi}_{2})-(\varphi'(\Phi_{2}^{(l-1)}
+\tilde{\Phi}_{2}))^{2}\Big)}{16\gamma\varphi^{2}(\Phi_{2}^{(l-1)}+\tilde{\Phi}_{2})}(\Phi_{3}^{(l-1)}
+\tilde{\Phi}_{3}-\Phi_{1}^{(l-1)}-\tilde{\Phi}_{1})^{2}\notag\\
&\cdot\mu_{1}(\Phi^{(l-1)}+\tilde{\Phi})\frac{\partial\Phi_{2}^{(l-1)}}{\partial x}\frac{\partial\Phi_{2}^{(l-1)}}{\partial t}\notag\\
&+\frac{\alpha(\gamma-1)(\tilde{\Phi}_{1}+\tilde{\Phi}_{3})}{(\gamma+1)\tilde{\Phi}_{1}
+(3-\gamma)\tilde{\Phi}_{3}}\Big(\frac{\partial\mu_{1}(\Phi^{(l-1)}+\tilde{\Phi})}{\partial\Phi_{1}^{(l-1)}}\Psi_{1}^{(l-1)}
+\frac{\partial\mu_{1}(\Phi^{(l-1)}+\tilde{\Phi})}{\partial\Phi_{3}^{(l-1)}}\Psi_{3}^{(l-1)}\Big)\Phi_{3}^{(l-1)}\notag\\
&+\frac{\alpha}{2}(1-\frac{(\gamma+1)(\tilde{\Phi}_{1}+\tilde{\Phi}_{3})}{(\gamma+1)\tilde{\Phi}_{1}
+(3-\gamma)\tilde{\Phi}_{3}})\Big(\mu_{1}(\Phi^{(l-1)}+\tilde{\Phi})-\mu_{1}(\tilde{\Phi})\Big)(\Psi_{1}^{(l-1)}+\Psi_{3}^{(l-1)})\notag\\
&-\frac{\varphi'(\Phi_{2}^{(l-1)}+\tilde{\Phi}_{2})}{4\gamma\varphi(\Phi_{2}^{(l-1)}+\tilde{\Phi}_{2})}(\Phi_{3}^{(l-1)}+\tilde{\Phi}_{3}
-\Phi_{1}^{(l-1)}-\tilde{\Phi}_{1})(\frac{\partial\mu_{1}(\Phi^{(l-1)}+\tilde{\Phi})}{\partial\Phi_{1}^{(l-1)}}\Psi_{1}^{(l-1)}\notag\\
&+\frac{\partial\mu_{1}(\Phi^{(l-1)}+\tilde{\Phi})}{\partial\Phi_{3}^{(l-1)}}\Psi_{3}^{(l-1)})\frac{\partial\Phi_{2}^{(l-1)}}
{\partial t}\notag\\
&+\frac{(\gamma-1)\varphi'(\Phi_{2}^{(l-1)}+\tilde{\Phi}_{2})}{16\gamma\varphi(\Phi_{2}^{(l-1)}+\tilde{\Phi}_{2})}\mu_{1}(\Phi^{(l-1)}
+\tilde{\Phi})(\partial_{t}\Phi_{3}^{(l-1)}-\partial_{t}\Phi_{1}^{(l-1)})\notag\\
&\cdot(\Phi_{3}^{(l-1)}+\tilde{\Phi}_{3}-\Phi_{1}^{(l-1)}-\tilde{\Phi}_{1})
\frac{\partial\Phi_{2}^{(l-1)}}{\partial x}\notag\\
&+\frac{\alpha}{2}(1-\frac{(\gamma+1)(\tilde{\Phi}_{1}+\tilde{\Phi}_{3})}{(\gamma+1)\tilde{\Phi}_{1}
+(3-\gamma)\tilde{\Phi}_{3}})\Big(\frac{\partial\mu_{1}(\Phi^{(l-1)}+\tilde{\Phi})}{\partial\Phi_{1}^{(l-1)}}\Psi_{1}^{(l-1)}\notag\\
&+\frac{\partial\mu_{1}(\Phi^{(l-1)}+\tilde{\Phi})}{\partial\Phi_{3}^{(l-1)}}\Psi_{3}^{(l-1)}\Big)(\Phi_{1}^{(l-1)}+\Phi_{3}^{(l-1)})\notag\\
=&-\frac{\varphi'(\Phi_{2}^{(l-1)}+\tilde{\Phi}_{2})}{4\gamma\varphi(\Phi_{2}^{(l-1)}+\tilde{\Phi}_{2})}(\Phi_{3}^{(l-1)}+\tilde{\Phi}_{3}
-\Phi_{1}^{(l-1)}-\tilde{\Phi}_{1})\Big(\frac{\partial\Psi_{2}^{(l-1)}}{\partial x}+\mu_{1}(\Phi^{(l-1)}+\tilde{\Phi})\frac{\partial\Psi_{2}^{(l-1)}}{\partial t}\Big)\notag\\
&+\frac{\alpha}{2}(1-\frac{(\gamma+1)(\tilde{\Phi}_{1}+\tilde{\Phi}_{3})}{(\gamma+1)\tilde{\Phi}_{1}+(3-\gamma)\tilde{\Phi}_{3}})
\mu_{1}(\tilde{\Phi})\Psi_{3}^{(l-1)}
+\frac{\alpha}{2}(1-\frac{(\gamma+1)(\tilde{\Phi}_{1}+\tilde{\Phi}_{3})}{(\gamma+1)\tilde{\Phi}_{1}+(3-\gamma)\tilde{\Phi}_{3}})
\mu_{1}(\tilde{\Phi})\Psi_{1}^{(l)}\notag\\
&-\Big(\frac{\partial\mu_{1}(\Phi^{(l-1)}+\tilde{\Phi})}{\partial\Phi_{1}^{(l-1)}}\Psi_{1}^{(l-1)}
+\frac{\partial\mu_{1}(\Phi^{(l-1)}+\tilde{\Phi})}{\partial\Phi_{3}^{(l-1)}}\Psi_{3}^{(l-1)}\Big)\Psi_{1}^{(l)}\notag\\
&+\frac{\alpha(\gamma-1)(\tilde{\Phi}_{1}+\tilde{\Phi}_{3})}{(\gamma+1)\tilde{\Phi}_{1}+(3-\gamma)\tilde{\Phi}}_{3}\mu_{1}(\Phi^{(l-1)}
+\tilde{\Phi})\Psi_{3}^{(l-1)}+O(\epsilon^{2}),\label{c57}\\
&x=L:~~\Psi_{1}^{(l)}(t,L)=\mathcal{H}'_{1}(t)+K_{1}\Psi_{3}^{(l-1)}(t,L),\label{c58}
\end{align}
\begin{align}
&\partial_{x}\Psi_{2}^{(l)}+\mu_{2}(\Phi^{(l-1)}+\tilde{\Phi})\partial_{t}\Psi_{2}^{(l)}
=-\Big(\frac{\partial\mu_{2}(\Phi^{(l-1)}+\tilde{\Phi})}{\partial\Phi_{1}^{(l-1)}}\Psi_{1}^{(l-1)}
+\frac{\partial\mu_{2}(\Phi^{(l-1)}+\tilde{\Phi})}{\partial\Phi_{3}^{(l-1)}}\Psi_{3}^{(l-1)}\Big)\Psi_{2}^{(l)},\label{c59}\\
&x=0:~~ \Psi_{2}^{(l)}(t,0)=\mathcal{H}'_{2}(t)+K_{2}\Psi_{1}^{(l)}(t,0).\label{c60}
\end{align}
Here there is a difficult point in this paper, which is that the equation~\eqref{c4} has the derivative loss. However, we find that the derivative of the entropy over $x$ can be written as the derivative along the $1$-st characteristic curve direction, this kind of method can overcome the derivative loss.
Specially, by~\eqref{b9}, $\mu_{1}(\Phi^{(l-1)}+\tilde{\Phi})=\lambda_{1}^{-1}(\Phi^{(l-1)}+\tilde{\Phi})$ and
\begin{align*}
&\lambda_{1}(\Phi^{(l-1)}+\tilde{\Phi})-\lambda_{2}(\Phi^{(l-1)}+\tilde{\Phi})\\
=&\frac{\gamma+1}{4}(\Phi_{1}^{(l-1)}+\tilde{\Phi}_{1})+\frac{3-\gamma}{4}(\Phi_{3}^{(l-1)}+\tilde{\Phi}_{3})-\frac{1}{2}(\Phi_{1}^{(l-1)}
+\tilde{\Phi}_{1}+\Phi_{3}^{(l-1)}+\tilde{\Phi}_{3})\\
=&-\frac{\gamma-1}{4}(\Phi_{3}^{(l-1)}+\tilde{\Phi}_{3}-\Phi_{1}^{(l-1)}-\tilde{\Phi}_{1}),
\end{align*}
we get
\begin{align}\label{CC1}
&\frac{(\gamma-1)\varphi'(\Phi_{2}^{(l-1)}+\tilde{\Phi}_{2})}{16\gamma\varphi(\Phi_{2}^{(l-1)}+\tilde{\Phi}_{2})}(\Phi_{3}^{(l-1)}
+\tilde{\Phi}_{3}-\Phi_{1}^{(l-1)}-\tilde{\Phi}_{1})^{2}\mu_{1}(\Phi^{(l-1)}+\tilde{\Phi})\frac{\partial\Phi_{2}^{(l-1)}}{\partial x}\notag\\
=&-\frac{\varphi'(\Phi_{2}^{(l-1)}+\tilde{\Phi}_{2})}{4\gamma\varphi(\Phi_{2}^{(l-1)}+\tilde{\Phi}_{2})}(\Phi_{3}^{(l-1)}
+\tilde{\Phi}_{3}-\Phi_{1}^{(l-1)}-\tilde{\Phi}_{1})(\frac{\partial\Phi_{2}^{(l-1)}}{\partial x}+\mu_{1}(\Phi^{(l-1)}+\tilde{\Phi})
\frac{\partial\Phi_{2}^{(l-1)}}{\partial t}).
\end{align}

At the boundary $x=L$, by~\eqref{b15}, \eqref{c31} and~\eqref{c49}, we have
\begin{align}
\|\Psi_{1}^{(l)}(\cdot,L)\|_{C^{0}(\mathbb{R})}\leq\epsilon+|K_{1}|M_{1}\epsilon\leq M_{1}\epsilon-99\mathcal{M}\epsilon.\label{c61}
\end{align}
Since the right side of~\eqref{c57} has the term related to $\Psi_{1}^{(l)}$ and this term cannot be estimated by the a prior estimate,
like the calculations in~\cite{Qup}, we introduce the method of "estimation by twice integration". We first get a rough estimate of $\Psi_{1}^{(l)}$ as follows.

We multiply $\mathrm{sgn}(\Psi_{1}^{(l)})$ on both sides of~\eqref{c57} to get
\begin{align}
&\partial_{x}|\Psi_{1}^{(l)}|+\mu_{1}(\Phi^{(l-1)}+\tilde{\Phi})\partial_{t}|\Psi_{1}^{(l)}|\notag\\
=&\Big(\frac{\alpha}{2}(1-\frac{(\gamma+1)(\tilde{\Phi}_{1}+\tilde{\Phi}_{3})}{(\gamma+1)\tilde{\Phi}_{1}+(3-\gamma)\tilde{\Phi}_{3}})
\mu_{1}(\tilde{\Phi})-\big(\frac{\partial\mu_{1}(\Phi^{(l-1)}+\tilde{\Phi})}{\partial\Phi_{1}^{(l-1)}}\Psi_{1}^{(l-1)}
+\frac{\partial\mu_{1}(\Phi^{(l-1)}+\tilde{\Phi})}{\partial\Phi_{3}^{(l-1)}}\Psi_{3}^{(l-1)}\big)\Big)|\Psi_{1}^{(l)}|\notag\\
%&+\frac{\alpha(\gamma-1)(\tilde{\Phi}_{1}+\tilde{\Phi}_{3})}{(\gamma+1)\tilde{\Phi}_{1}
%+(3-\gamma)\tilde{\Phi}_{3}}\mathrm{sgn}(\Psi_{1}^{(l)})\Big(\frac{\partial\mu_{1}(\Phi^{(l-1)}+\tilde{\Phi})}{\partial\Phi_{1}^{(l-1)}}\Psi_{1}^{(l-1)}
%+\frac{\partial\mu_{1}(\Phi^{(l-1)}+\tilde{\Phi})}{\partial\Phi_{3}^{(l-1)}}\Psi_{3}^{(l-1)}\Big)\Phi_{3}^{(l-1)}\notag\\
%&+\frac{(\gamma-1)\Big(\varphi(\Phi_{2}^{(l-1)}+\tilde{\Phi}_{2})\varphi''(\Phi_{2}^{(l-1)}+\tilde{\Phi}_{2})-(\varphi'(\Phi_{2}^{(l-1)}
%+\tilde{\Phi}_{2}))^{2}\Big)}{16\gamma\varphi^{2}(\Phi_{2}^{(l-1)}+\tilde{\Phi}_{2})}(\Phi_{3}^{(l-1)}
%+\tilde{\Phi}_{3}-\Phi_{1}^{(l-1)}-\tilde{\Phi}_{1})^{2}\notag\\
%&\cdot\mathrm{sgn}(\Psi_{1}^{(l)})\mu_{1}(\Phi^{(l-1)}+\tilde{\Phi})\frac{\partial\Phi_{2}^{(l-1)}}{\partial x}\frac{\partial\Phi_{2}^{(l-1)}}{\partial t}\notag\\
&-\frac{\varphi'(\Phi_{2}^{(l-1)}+\tilde{\Phi}_{2})}{4\gamma\varphi(\Phi_{2}^{(l-1)}+\tilde{\Phi}_{2})}\mathrm{sgn}(\Psi_{1}^{(l)})(\Phi_{3}^{(l-1)}+\tilde{\Phi}_{3}
-\Phi_{1}^{(l-1)}-\tilde{\Phi}_{1})(\frac{\partial\Psi_{2}^{(l-1)}}{\partial x}+\mu_{1}(\Phi^{(l-1)}+\tilde{\Phi})\frac{\partial\Psi_{2}^{(l-1)}}{\partial t})\notag\\
%&+\frac{\alpha}{2}(1-\frac{(\gamma+1)(\tilde{\Phi}_{1}+\tilde{\Phi}_{3})}{(\gamma+1)\tilde{\Phi}_{1}
%+(3-\gamma)\tilde{\Phi}_{3}})\mathrm{sgn}(\Psi_{1}^{(l)})\Big(\mu_{1}(\Phi^{(l-1)}+\tilde{\Phi})-\mu_{1}(\tilde{\Phi})\Big)(\Psi_{1}^{(l-1)}+\Psi_{3}^{(l-1)})\notag\\
%&+\frac{(\gamma-1)\varphi'(\Phi_{2}^{(l-1)}+\tilde{\Phi}_{2})}{16\gamma\varphi(\Phi_{2}^{(l-1)}+\tilde{\Phi}_{2})}
%(\partial_{t}\Phi_{3}^{(l-1)}-\partial_{t}\Phi_{1}^{(l-1)})(\Phi_{3}^{(l-1)}+\tilde{\Phi}_{3}-\Phi_{1}^{(l-1)}-\tilde{\Phi}_{1})\notag\\
%&\cdot\mathrm{sgn}(\Psi_{1}^{(l)})\mu_{1}(\Phi^{(l-1)}+\tilde{\Phi})\frac{\partial\Phi_{2}^{(l-1)}}{\partial x}\notag\\
%&-\frac{\varphi'(\Phi_{2}^{(l-1)}+\tilde{\Phi}_{2})}{4\gamma\varphi(\Phi_{2}^{(l-1)}+\tilde{\Phi}_{2})}(\Phi_{3}^{(l-1)}+\tilde{\Phi}_{3}
%-\Phi_{1}^{(l-1)}-\tilde{\Phi}_{1})(\frac{\partial\mu_{1}(\Phi^{(l-1)}+\tilde{\Phi})}{\partial\Phi_{1}^{(l-1)}}\Psi_{1}^{(l-1)}\notag\\
%&+\frac{\partial\mu_{1}(\Phi^{(l-1)}+\tilde{\Phi})}{\partial\Phi_{3}^{(l-1)}}\Psi_{3}^{(l-1)})\mathrm{sgn}(\Psi_{1}^{(l)})\frac{\partial\Phi_{2}^{(l-1)}}
%{\partial t}\notag\\
%&+\frac{\alpha}{2}(1-\frac{(\gamma+1)(\tilde{\Phi}_{1}+\tilde{\Phi}_{3})}{(\gamma+1)\tilde{\Phi}_{1}
%+(3-\gamma)\tilde{\Phi}_{3}})\mathrm{sgn}(\Psi_{1}^{(l)})\Big(\frac{\partial\mu_{1}(\Phi^{(l-1)}+\tilde{\Phi})}{\partial\Phi_{1}^{(l-1)}}\Psi_{1}^{(l-1)}\notag\\
%&+\frac{\partial\mu_{1}(\Phi^{(l-1)}+\tilde{\Phi})}{\partial\Phi_{3}^{(l-1)}}\Psi_{3}^{(l-1)}\Big)(\Phi_{1}^{(l-1)}+\Phi_{3}^{(l-1)})\notag\\
&+\frac{\alpha(\gamma-1)(\tilde{\Phi}_{1}+\tilde{\Phi}_{3})}{(\gamma+1)\tilde{\Phi}_{1}+(3-\gamma)\tilde{\Phi}_{3}}
\mathrm{sgn}(\Psi_{1}^{(l)})\mu_{1}(\Phi^{(l-1)}+\tilde{\Phi})\Psi_{3}^{(l-1)}\notag\\
&+\frac{\alpha}{2}(1-\frac{(\gamma+1)(\tilde{\Phi}_{1}+\tilde{\Phi}_{3})}{(\gamma+1)\tilde{\Phi}_{1}+(3-\gamma)\tilde{\Phi}_{3}})
\mathrm{sgn}(\Psi_{1}^{(l)})\mu_{1}(\tilde{\Phi})\Psi_{3}^{(l-1)}+O(\epsilon^{2}),\label{c62}
\end{align}
By~\eqref{a2}, \eqref{A3}, \eqref{B13}-\eqref{a4}, \eqref{c29} and~\eqref{c31}, there exists a constant $C>0$ such that
$$\frac{\alpha}{2}(1-\frac{(\gamma+1)(\tilde{\Phi}_{1}+\tilde{\Phi}_{3})}{(\gamma+1)\tilde{\Phi}_{1}+(3-\gamma)\tilde{\Phi}_{3}})
\mu_{1}(\tilde{\Phi})-\big(\frac{\partial\mu_{1}(\Phi^{(l-1)}+\tilde{\Phi})}{\partial\Phi_{1}^{(l-1)}}\Psi_{1}^{(l-1)}
+\frac{\partial\mu_{1}(\Phi^{(l-1)}+\tilde{\Phi})}{\partial\Phi_{3}^{(l-1)}}\Psi_{3}^{(l-1)}\big)>-C\epsilon,$$
then integrating~\eqref{c62} along the 1-st characteristic curve $t=t_{1}^{(l)}(y;t,x)$ from $L$ to $x$ and using~\eqref{a2}, \eqref{A3}, \eqref{B13}-\eqref{a4}, \eqref{c29}-\eqref{c31} and~\eqref{c61}, we get
\begin{align}
|\Psi_{1}^{(l)}(t,x)|\leq C\epsilon,\quad \forall (t,x)\in\mathbb{R}\times[0,L].\label{c63}
\end{align}
Next, we change equation~\eqref{c57} into the following equation of $F_{1}\Psi_{1}^{(l)}$
\begin{align}
&\partial_{x}(F_{1}\Psi_{1}^{(l)})+\mu_{1}(\Phi^{(l-1)}+\tilde{\Phi})\partial_{t}(F_{1}\Psi_{1}^{(l)})\notag\\
=&-\frac{\varphi'(\Phi_{2}^{(l-1)}+\tilde{\Phi}_{2})}{4\gamma\varphi(\Phi_{2}^{(l-1)}+\tilde{\Phi}_{2})}F_{1}(\Phi_{3}^{(l-1)}+\tilde{\Phi}_{3}
-\Phi_{1}^{(l-1)}-\tilde{\Phi}_{1})(\frac{\partial\Psi_{2}^{(l-1)}}{\partial x}+\mu_{1}(\Phi^{(l-1)}+\tilde{\Phi})\frac{\partial\Psi_{2}^{(l-1)}}{\partial t})\notag\\
&+\frac{\alpha}{2}(1-\frac{(\gamma+1)(\tilde{\Phi}_{1}+\tilde{\Phi}_{3})}{(\gamma+1)\tilde{\Phi}_{1}+(3-\gamma)\tilde{\Phi}_{3}})F_{1}
\mu_{1}(\tilde{\Phi})\Psi_{3}^{(l-1)}
+\frac{\alpha(\gamma-1)(\tilde{\Phi}_{1}+\tilde{\Phi}_{3})}{(\gamma+1)\tilde{\Phi}_{1}+(3-\gamma)\tilde{\Phi}}_{3}F_{1}\mu_{1}(\Phi^{(l-1)}
+\tilde{\Phi})\Psi_{3}^{(l-1)}\notag\\
&-F_{1}\Big(\frac{\partial\mu_{1}(\Phi^{(l-1)}+\tilde{\Phi})}{\partial\Phi_{1}^{(l-1)}}\Psi_{1}^{(l-1)}
+\frac{\partial\mu_{1}(\Phi^{(l-1)}+\tilde{\Phi})}{\partial\Phi_{3}^{(l-1)}}\Psi_{3}^{(l-1)}\Big)\Psi_{1}^{(l)}+O(\epsilon^{2})
,\label{C64}
\end{align}
then we integrate it along the 1-st characteristic curve $t=t_{1}^{(l)}(y;t,x)$ from $L$ to $x$ and use~\eqref{a2}, \eqref{A3}, \eqref{B13}-\eqref{a4}, \eqref{c29}-\eqref{c31}, \eqref{c42}, \eqref{c61} and~\eqref{c63} to get
\begin{align}
\|\Psi_{1}^{(l)}\|_{C^{0}}\leq&\frac{M_{1}\epsilon-99\mathcal{M}\epsilon}{F_{1}(x)}+C\epsilon^{2}
+\frac{F_{1}(x)-1}{F_{1}(x)}M_{1}\epsilon+C\epsilon_{0}\epsilon\notag\\
=&M_{1}\epsilon-\frac{99\mathcal{M}\epsilon}{F_{1}(x)}+C\epsilon^{2}+C\epsilon_{0}\epsilon\notag\\
<&\hbar_{2}M_{1}\epsilon,\label{c64}
\end{align}
where the constant $0<\hbar_{2}<1$.

By applying the equation~\eqref{c4} and using~\eqref{a2}, \eqref{A3}, \eqref{B13}-\eqref{a4}, \eqref{c29}-\eqref{c31}, \eqref{c54} and~\eqref{c64}, one has
\begin{align}
\|\partial_{x}\Phi_{1}^{(l)}\|_{C^{0}}\leq&\mathcal{K}\hbar_{2}M_{1}\epsilon-\alpha_{*}(1+\frac{\gamma+1}{2}\mathcal{K}c_{-})\mathcal{K}M_{1}\epsilon
+C\epsilon^{2}+C\epsilon_{0}\epsilon\notag\\
\leq& M_{2}\epsilon,\label{c65}
\end{align}
where the constant $M_{2}>\mathcal{K}\hbar_{2}M_{1}-\alpha_{*}(1+\frac{\gamma+1}{2}\mathcal{K}c_{-})\mathcal{K}M_{1}$. This means that we complete the proof of the estimates~\eqref{c18}.

At the boundary $x=0$, by~\eqref{B15} and~\eqref{c64}, we have
\begin{align}
|\Psi_{2}^{(l)}(t,0)|\leq \epsilon+|K_{2}|\hbar_{2}M_{1}\epsilon,\label{c66}
\end{align}
In the domain $\mathcal{D}=\{(t,x)|t\in\mathbb{R},x\in[0,L]\}$, using~\eqref{c29}, \eqref{c31} and~\eqref{c66}, we get
\begin{align}
\|\Psi_{2}^{(l)}(t,x)\|_{C^{0}}\leq e^{-C\epsilon L}(\epsilon+|K_{2}|\hbar_{2}M_{1}\epsilon)\leq M_{1}\epsilon.\label{c67}
\end{align}
With the aid of the equation~\eqref{c6}, the hypothesis~\eqref{B13} and the estimates~\eqref{c67}, one has
\begin{align}
\|\partial_{x}\Phi_{2}^{(l)}\|_{C^{0}}\leq \mathcal{K}M_{1}\epsilon\leq M_{2}\epsilon.\label{c68}
\end{align}
The proof of~\eqref{c19} complete.

Now, we start the proof of~\eqref{c21}. First, we get directly from~\eqref{c10} and~\eqref{c54}
\begin{align}
\|\Phi_{1}^{(1)}(t,x)-\Phi_{1}^{(0)}(t,x)\|_{C^{0}}\leq\hbar_{1}M_{1}\epsilon<M_{1}\epsilon\eta.\label{LL1}
\end{align}
Then, by~\eqref{B15}, \eqref{c10} and~\eqref{c56}, we have
\begin{align}
\|\Phi_{2}^{(1)}(t,x)-\Phi_{2}^{(0)}(t,x)\|_{C^{0}}=\|\Phi_{2}^{(1)}(t,0)-M_{0}\|_{C^{0}}
=\|\mathcal{H}_{2}(t)+K_{2}\Phi_{1}^{(0)}(t,0)-M_{0}\|_{C^{0}}\leq \epsilon<M_{1}\epsilon\eta.\label{LL2}
\end{align}
Next we prove~\eqref{c21} for $l\geq2$.
Select $\eta<1$ satisfying
$$ \eta>K.$$
At the boundary $x=L$, by~\eqref{c5} and~\eqref{c32}, we have
\begin{align}
\|\Phi_{1}^{(l)}(t,L)-\Phi_{1}^{(l-1)}(t,L)\|_{C^{0}}\leq&|K_{1}|\|\Phi_{3}^{(l-1)}(t,L)-\Phi_{3}^{(l-2)}(t,L)\|_{C^{0}}\notag\\
\leq&|K_{1}|M_{1}\epsilon\eta^{l-1}.\label{c69}
\end{align}
From the equations~\eqref{c4}, we get
\begin{align}
&\partial_{x}(\Phi_{1}^{(l)}-\Phi_{1}^{(l-1)})+\mu_{1}(\Phi^{(l-1)}+\tilde{\Phi})\partial_{t}(\Phi_{1}^{(l)}-\Phi_{1}^{(l-1)})\notag\\
=&-\frac{\varphi'(\Phi_{2}^{(l-2)}+\tilde{\Phi}_{2})}{4\gamma\varphi(\Phi_{2}^{(l-2)}+\tilde{\Phi}_{2})}(\Phi_{3}^{(l-2)}
+\tilde{\Phi}_{3}-\Phi_{1}^{(l-2)}-\tilde{\Phi}_{1})\Big(\frac{\partial}{\partial x}(\Phi_{2}^{(l-1)}-\Phi_{2}^{(l-2)})\notag\\
&+\mu_{1}(\Phi^{(l-1)}+\tilde{\Phi})
\frac{\partial}{\partial t}(\Phi_{2}^{(l-1)}-\Phi_{2}^{(l-2)})\Big)\notag\\
&+\frac{\alpha}{2}(1-\frac{(\gamma+1)(\tilde{\Phi}_{1}+\tilde{\Phi}_{3})}{(\gamma+1)\tilde{\Phi}_{1}+(3-\gamma)\tilde{\Phi}_{3}})\mu_{1}
(\tilde{\Phi})(\Phi_{3}^{(l-1)}-\Phi_{3}^{(l-2)})\notag\\
&+\frac{\alpha(\gamma-1)(\tilde{\Phi}_{1}+\tilde{\Phi}_{3})}{(\gamma+1)\tilde{\Phi}_{1}+(3-\gamma)\tilde{\Phi}_{3}}\mu_{1}(\Phi^{(l-1)}
+\tilde{\Phi})(\Phi_{3}^{(l-1)}-\Phi_{3}^{(l-2)})\notag\\
&+\frac{\alpha}{2}(1-\frac{(\gamma+1)(\tilde{\Phi}_{1}+\tilde{\Phi}_{3})}{(\gamma+1)\tilde{\Phi}_{1}+(3-\gamma)\tilde{\Phi}_{3}})\mu_{1}
(\tilde{\Phi})(\Phi_{1}^{(l)}-\Phi_{1}^{(l-1)})\notag\\
&-\frac{\varphi'(\Phi_{2}^{(l-2)}+\tilde{\Phi}_{2})}{4\gamma\varphi(\Phi_{2}^{(l-2)}+\tilde{\Phi}_{2})}(\Phi_{3}^{(l-2)}
+\tilde{\Phi}_{3}-\Phi_{1}^{(l-2)}-\tilde{\Phi}_{1})\Big(\mu_{1}(\Phi^{(l-1)}+\tilde{\Phi})\notag\\
&-\mu_{1}(\Phi^{(l-2)}+\tilde{\Phi})\Big)
\frac{\partial\Phi_{2}^{(l-2)}}{\partial t}\notag\\
&+\frac{(\gamma-1)\varphi'(\Phi_{2}^{(l-2)}+\tilde{\Phi}_{2})}{16\gamma\varphi(\Phi_{2}^{(l-2)}+\tilde{\Phi}_{2})}\Big((\Phi_{3}^{(l-1)}
-\Phi_{3}^{(l-2)})-(\Phi_{1}^{(l-1)}-\Phi_{1}^{(l-2)})\Big)\notag\\
&\cdot(\Phi_{3}^{(l-1)}
+\tilde{\Phi}_{3}-\Phi_{1}^{(l-1)}-\tilde{\Phi}_{1})
\mu_{1}(\Phi^{(l-1)}+\tilde{\Phi})\frac{\partial\Phi_{2}^{(l-1)}}{\partial x}\notag\\
&+\frac{\alpha}{2}(1-\frac{(\gamma+1)(\tilde{\Phi}_{1}+\tilde{\Phi}_{3})}{(\gamma+1)\tilde{\Phi}_{1}+(3-\gamma)\tilde{\Phi}_{3}})\Big
(\mu_{1}(\Phi^{(l-1)}+\tilde{\Phi})-\mu_{1}(\Phi^{(l-2)}+\tilde{\Phi})\Big)\Phi_{1}^{(l-2)}\notag\\
&+\frac{\alpha}{2}(1-\frac{(3-\gamma)(\tilde{\Phi}_{1}+\tilde{\Phi}_{3})}{(\gamma+1)\tilde{\Phi}_{1}+(3-\gamma)\tilde{\Phi}_{3}})
\Big(\mu_{1}(\Phi^{(l-1)}+\tilde{\Phi})-\mu_{1}(\Phi^{(l-2)}+\tilde{\Phi})\Big)\Phi_{3}^{(l-2)}\notag\\
&+\frac{\gamma-1}{16\gamma}(\frac{\varphi'(\Phi_{2}^{(l-1)}+\tilde{\Phi}_{2})}{\varphi(\Phi_{2}^{(l-1)}+\tilde{\Phi}_{2})}
-\frac{\varphi'(\Phi_{2}^{(l-2)}+\tilde{\Phi}_{2})}{\varphi(\Phi_{2}^{(l-2)}+\tilde{\Phi}_{2})})(\Phi_{3}^{(l-1)}
+\tilde{\Phi}_{3}-\Phi_{1}^{(l-1)}-\tilde{\Phi}_{1})^{2}\notag\\
&\cdot\mu_{1}(\Phi^{(l-1)}+\tilde{\Phi})
\frac{\partial\Phi_{2}^{(l-1)}}{\partial x}\notag\\
&+\frac{\alpha}{2}(1-\frac{(\gamma+1)(\tilde{\Phi}_{1}+\tilde{\Phi}_{3})}{(\gamma+1)\tilde{\Phi}_{1}+(3-\gamma)\tilde{\Phi}_{3}})\Big
(\mu_{1}(\Phi^{(l-1)}+\tilde{\Phi})-\mu_{1}(\tilde{\Phi})\Big)\notag\\
&\cdot\Big((\Phi_{1}^{(l-1)}+\Phi_{3}^{(l-1)})
-(\Phi_{1}^{(l-2)}+\Phi_{3}^{(l-2)})\Big)\notag\\
&-\Big(\mu_{1}(\Phi^{(l-1)}+\tilde{\Phi})-\mu_{1}(\Phi^{(l-2)}+\tilde{\Phi})\Big)\partial_{t}\Phi_{1}^{(l-1)},\label{c70}
\end{align}
then multiplying $F_{1}(x)$ on both sides of the equations~\eqref{c70} and integrating along the 1-st characteristic curve $t=t_{1}^{(l)}(y;t,x)$ from $L$ to $x$, we get
\begin{align}
&F_{1}(x)\Big(\Phi_{1}^{(l)}(t,x)-\Phi_{1}^{(l-1)}(t,x)\Big)\notag\\
=&F_{1}(L)\Big(\Phi_{1}^{(l)}(t_{1}^{(l)}(L;t,x),L)-\Phi_{1}^{(l-1)}(t_{1}^{(l)}(L;t,x),L)\Big)\notag\\
&+\int_{L}^{x}\Big(
-\frac{\varphi'(\Phi_{2}^{(l-2)}+\tilde{\Phi}_{2})}{4\gamma\varphi(\Phi_{2}^{(l-2)}+\tilde{\Phi}_{2})}F_{1}(\Phi_{3}^{(l-2)}
+\tilde{\Phi}_{3}-\Phi_{1}^{(l-2)}-\tilde{\Phi}_{1})\Big(\frac{\partial}{\partial x}(\Phi_{2}^{(l-1)}-\Phi_{2}^{(l-2)})\notag\\
&+\mu_{1}(\Phi^{(l-1)}+\tilde{\Phi})
\frac{\partial}{\partial t}(\Phi_{2}^{(l-1)}-\Phi_{2}^{(l-2)})\Big)\notag\\
&+\frac{\alpha(\gamma-1)(\tilde{\Phi}_{1}+\tilde{\Phi}_{3})}{(\gamma+1)\tilde{\Phi}_{1}+(3-\gamma)\tilde{\Phi}_{3}}F_{1}\mu_{1}(\Phi^{(l-1)}
+\tilde{\Phi})(\Phi_{3}^{(l-1)}-\Phi_{3}^{(l-2)})\notag\\
&+\frac{\alpha}{2}(1-\frac{(\gamma+1)(\tilde{\Phi}_{1}+\tilde{\Phi}_{3})}{(\gamma+1)\tilde{\Phi}_{1}+(3-\gamma)\tilde{\Phi}_{3}})F_{1}\big
(\mu_{1}(\Phi^{(l-1)}+\tilde{\Phi})-\mu_{1}(\Phi^{(l-2)}+\tilde{\Phi})\big)\Phi_{1}^{(l-2)}\notag\\
&+\frac{\alpha}{2}(1-\frac{(3-\gamma)(\tilde{\Phi}_{1}+\tilde{\Phi}_{3})}{(\gamma+1)\tilde{\Phi}_{1}+(3-\gamma)\tilde{\Phi}_{3}})F_{1}
\big(\mu_{1}(\Phi^{(l-1)}+\tilde{\Phi})-\mu_{1}(\Phi^{(l-2)}+\tilde{\Phi})\big)\Phi_{3}^{(l-2)}\notag\\
&+\frac{\gamma-1}{16\gamma}(\frac{\varphi'(\Phi_{2}^{(l-1)}+\tilde{\Phi}_{2})}{\varphi(\Phi_{2}^{(l-1)}+\tilde{\Phi}_{2})}
-\frac{\varphi'(\Phi_{2}^{(l-2)}+\tilde{\Phi}_{2})}{\varphi(\Phi_{2}^{(l-2)}+\tilde{\Phi}_{2})})F_{1}(\Phi_{3}^{(l-1)}
+\tilde{\Phi}_{3}-\Phi_{1}^{(l-1)}-\tilde{\Phi}_{1})^{2}\notag\\
&\cdot\mu_{1}(\Phi^{(l-1)}+\tilde{\Phi})
\frac{\partial\Phi_{2}^{(l-1)}}{\partial x}\notag\\
&-\frac{\varphi'(\Phi_{2}^{(l-2)}+\tilde{\Phi}_{2})}{4\gamma\varphi(\Phi_{2}^{(l-2)}+\tilde{\Phi}_{2})}F_{1}(\Phi_{3}^{(l-2)}
+\tilde{\Phi}_{3}-\Phi_{1}^{(l-2)}-\tilde{\Phi}_{1})\Big(\mu_{1}(\Phi^{(l-1)}+\tilde{\Phi})\notag\\
&-\mu_{1}(\Phi^{(l-2)}+\tilde{\Phi})\Big)
\frac{\partial\Phi_{2}^{(l-2)}}{\partial t}\notag\\
&+\frac{(\gamma-1)\varphi'(\Phi_{2}^{(l-2)}+\tilde{\Phi}_{2})}{16\gamma\varphi(\Phi_{2}^{(l-2)}+\tilde{\Phi}_{2})}F_{1}\Big((\Phi_{3}^{(l-1)}
-\Phi_{3}^{(l-2)})-(\Phi_{1}^{(l-1)}-\Phi_{1}^{(l-2)})\Big)\notag\\
&\cdot(\Phi_{3}^{(l-1)}
+\tilde{\Phi}_{3}-\Phi_{1}^{(l-1)}-\tilde{\Phi}_{1})\mu_{1}(\Phi^{(l-1)}+\tilde{\Phi})\frac{\partial\Phi_{2}^{(l-1)}}{\partial x}\notag\\
&+\frac{\alpha}{2}(1-\frac{(\gamma+1)(\tilde{\Phi}_{1}+\tilde{\Phi}_{3})}{(\gamma+1)\tilde{\Phi}_{1}+(3-\gamma)\tilde{\Phi}_{3}})F_{1}\big
(\mu_{1}(\Phi^{(l-1)}+\tilde{\Phi})-\mu_{1}(\tilde{\Phi})\big)\notag\\
&\cdot\big((\Phi_{1}^{(l-1)}+\Phi_{3}^{(l-1)})
-(\Phi_{1}^{(l-2)}+\Phi_{3}^{(l-2)})\big)\notag\\
&-F_{1}\big(\mu_{1}(\Phi^{(l-1)}+\tilde{\Phi})-\mu_{1}(\Phi^{(l-2)}+\tilde{\Phi})\big)\partial_{t}\Phi_{1}^{(l-1)}\Big)(t_{1}^{(l)}(y;t,x),y)dy\notag\\
&+\int_{L}^{x}-(\frac{d}{dy}F_{1}(y)(\Phi_{3}^{(l-1)}-\Phi_{3}^{(l-2)})(t_{1}^{(l)}(y;t,x),y)dy.\label{c71}
\end{align}
Using~\eqref{a2}, \eqref{A3}, \eqref{B13}-\eqref{a4}, \eqref{c29}-\eqref{c32}, \eqref{c42} and~\eqref{c69}, we have
\begin{align}
\|\Phi_{1}^{(l)}-\Phi_{1}^{(l-1)}\|_{C^{0}}\leq&\frac{|K_{1}|M_{1}\epsilon\eta^{l-1}}{F_{1}(x)}+\frac{F_{1}(x)-1}{F_{1}(x)}M_{1}
\epsilon\eta^{l-1}+C\epsilon M_{1}\epsilon\eta^{l-1}+C\epsilon_{0} M_{1}\epsilon\eta^{l-1}\notag\\
\leq&\hbar_{3}M_{1}\epsilon\eta^{l},\label{c72}
\end{align}
where we assume $\hbar_{3}\eta>\frac{F_{1}(x)-1+|K_{1}|}{F_{1}(x)}$ with the constant $0<\hbar_{3}<1$.

In addition, at the boundary $x=0$, it follows from~\eqref{c72} that
\begin{align}
\|\Phi_{2}^{(l)}(t,0)-\Phi_{2}^{(l-1)}(t,0)\|_{C^{0}}\leq |K_{2}|\hbar_{3}M_{1}\epsilon\eta^{l}.\label{c73}
\end{align}
In the domain $\mathcal{D}$, from the equations~\eqref{c6}, we get
\begin{align*}
&\partial_{x}(\Phi_{2}^{(l)}-\Phi_{2}^{(l-1)})+\mu_{2}(\Phi^{(l-1)}+\tilde{\Phi})\partial_{t}(\Phi_{2}^{(l)}-\Phi_{2}^{(l-1)})\\
=&-\big(\mu_{2}(\Phi^{(l-1)}+\tilde{\Phi})-\mu_{2}(\Phi^{(l-2)}+\tilde{\Phi})\big)\partial_{t}\Phi_{2}^{(l-1)},
\end{align*}
then by~\eqref{c30}, \eqref{c32} and~\eqref{c73}, we have
\begin{align}
\|\Phi_{2}^{(l)}-\Phi_{2}^{(l-1)}\|_{C^{0}}\leq &|K_{2}|\hbar_{3}M_{1}\epsilon\eta^{l}+C\epsilon M_{1}\epsilon\eta^{l-1}\notag\\
<&M_{1}\epsilon\eta^{l}.\label{c74}
\end{align}
We complete the proof of~\eqref{c21}.

Now, we show the modulus of continuity for $\Psi_{i}^{(l)}(i=1,2,3)$ on the temporal direction~\eqref{c23}, which is very important to prove~\eqref{c25}.\\
\indent For $\delta\in(0,1)$, we choose
\begin{align}
\Omega(\delta)=\frac{24}{1-K}[\mathcal{K}+1]\Big(\sqrt{\epsilon}\delta+\varpi(\delta|\mathcal{H}'_{1})
+\varpi(\delta|\mathcal{H}'_{2})+\varpi(\delta|\mathcal{H}'_{3})\Big).\label{c75}
\end{align}
Since $\varpi(\delta|\mathcal{H}'_{i})(i=1,2,3)$ are monotonically increasing, bounded and continuous concave functions of $\delta$ and $\mathop{\lim}\limits_{\delta\rightarrow0^{+}}\varpi(\delta|\mathcal{H}'_{i})=0$, then $\Omega(\delta)$ has the same features and
$$\mathop{\lim}\limits_{\delta\rightarrow0^{+}}\Omega(\delta)=0.$$
At the boundary $x=L$, for any given $t_{1},t_{2}\in\mathbb{R}$ with $|t_{1}-t_{2}|\leq\delta\ll1$, one has
\begin{align*}
|\Psi_{1}^{(l)}(t_{1},L)-\Psi_{1}^{(l)}(t_{2},L)|\leq|\mathcal{H}'_{1}(t_{1})-\mathcal{H}'_{1}(t_{2})|+|K_{1}||\Psi_{3}^{(l-1)}(t_{1},L)
-\Psi_{3}^{(l-1)}(t_{2},L)|,
\end{align*}
then by~\eqref{c34} and~\eqref{c75}, we have
\begin{align}
\varpi(\delta|\Psi_{1}^{(l)}(\cdot,L))\leq&\varpi(\delta|\mathcal{H}'_{1})+|K_{1}|\varpi(\delta|\Psi_{3}^{(l-1)}(\cdot,L))\notag\\
\leq&\frac{1-K}{24[\mathcal{K}+1]}\Omega(\delta)+\frac{|K_{1}|}{8[\mathcal{K}+1]}\Omega(\delta)\notag\\
\leq&(\frac{1}{24[\mathcal{K}+1]}+\frac{K}{12[\mathcal{K}+1]})\Omega(\delta)\notag\\
<&\frac{1}{8[\mathcal{K}+1]}\Omega(\delta).\label{c76}
\end{align}
In the domain $\mathcal{D}$, for any $x\in[0,L]$ and $t_{1},t_{2}\in\mathbb{R}$ with $|t_{1}-t_{2}|\leq\delta$, by the definition of the characteristic curve, one has
$$
t_{1}^{(l)}(y;t_{*},x)=\int_{x}^{y}\mu_{1}(\Phi^{(l-1)}(t_{1}^{(l)}(\tilde{y};t_{*},x),\tilde{y})+\tilde{\Phi})d\tilde{y}+t_{*}.
$$
As a consequence,
\begin{align*}
&|t_{1}^{(l)}(y;t_{1},x)-t_{1}^{(l)}(y;t_{2},x)|\notag\\
\leq&|t_{1}-t_{2}|+\int_{x}^{y}|\mu_{1}(\Phi^{(l-1)}(t_{1}^{(l)}(\tilde{y};t_{1},x),\tilde{y})+\tilde{\Phi})
-\mu_{1}(\Phi^{(l-1)}(t_{1}^{(l)}(\tilde{y};t_{2},x),\tilde{y})+\tilde{\Phi})|d\tilde{y}\notag\\
\leq&|t_{1}-t_{2}|+\int_{x}^{y}(|\frac{\partial\mu_{1}}{\partial\Phi_{1}}|+|\frac{\partial\mu_{1}}{\partial\Phi_{3}}|)\|
\Psi_{j}^{(l-1)}\|_{C^{0}}|t_{1}^{(l)}(\tilde{y};t_{1},x)-t_{1}^{(l)}(\tilde{y};t_{2},x)|d\tilde{y},\notag
\end{align*}
and by the Gronwall's inequality, \eqref{c29} and~\eqref{c31}, we have
\begin{align}
&|t_{1}^{(l)}(y;t_{1},x)-t_{1}^{(l)}(y;t_{2},x)|\notag\\
\leq& e^{C\epsilon}|t_{1}-t_{2}|\leq(1+\sqrt{\epsilon})\delta.\label{c77}
\end{align}
By the concavity of $\Omega(\delta)$, we have
$$\frac{1}{1+\sqrt{\epsilon}}\Omega((1+\sqrt{\epsilon})\delta)+\frac{\sqrt{\epsilon}}{1+\sqrt{\epsilon}}\Omega(0)\leq \Omega(\delta),$$
which indicates
$$\Omega((1+\sqrt{\epsilon})\delta)\leq(1+\sqrt{\epsilon})\Omega(\delta).$$
Thus, by~\eqref{c34}, we get
\begin{align}
&|\Psi_{i}^{(l-1)}(t_{1}^{(l)}(y;t_{1},x),y)-\Psi_{i}^{(l-1)}(t_{1}^{(l)}(y;t_{2},x),y)|\notag\\
\leq &\frac{1}{8[\mathcal{K}+1]}\Omega((1+\sqrt{\epsilon})\delta)\notag\\
\leq&\frac{1}{8[\mathcal{K}+1]}(1+\sqrt{\epsilon})\Omega(\delta),\quad i=1,2,3.\label{c78}
\end{align}

Next, we integrate~\eqref{c57} along the 1-st characteristic $t=t_{1}^{(l)}(y;t_{1},x)$ and $t=t_{1}^{(l)}(y;t_{2},x)$ respectively and use the Gronwall's inequality, \eqref{a2}, \eqref{A3}, \eqref{B13}-\eqref{a4}, \eqref{c29}-\eqref{c31} and~\eqref{c76}-\eqref{c78} to get
\begin{align}
|\Psi_{1}^{(l)}(t_{1},x)-\Psi_{1}^{(l)}(t_{2},x)|\leq\frac{1}{8[\mathcal{K}+1]}C\Omega(\delta),\label{c80}
\end{align}
where $C>0$ is a generic constant.

Then using the integral expression~\eqref{C64} of $F_{1}(x)\Psi_{1}^{(l)}(t,x)$ and by~\eqref{a2}, \eqref{A3}, \eqref{B13}-\eqref{a4}, \eqref{c29}-\eqref{c31}, \eqref{c42} and~\eqref{c76}-\eqref{c80}, we have
\begin{align}
\varpi(\delta|\Psi_{1}^{(l)}(\cdot,x))
\leq&\frac{1+\sqrt{\epsilon}}{F_{1}(x)}(\frac{1}{24[\mathcal{K}+1]}+\frac{K}{12[\mathcal{K}+1]})\Omega(\delta)
+\frac{F_{1}(x)-1}{F_{1}(x)}\frac{1}{8[\mathcal{K}+1]}(1+\sqrt{\epsilon})\Omega(\delta)\notag\\
&+C\epsilon(1+\sqrt{\epsilon})\delta
+C\epsilon\frac{1+\sqrt{\epsilon}}{8[\mathcal{K}+1]}\Omega(\delta)
+C\epsilon_{0}\frac{1+\sqrt{\epsilon}}{8[\mathcal{K}+1]}\Omega(\delta)\notag\\
<&\frac{1}{8[\mathcal{K}+1]}\Omega(\delta)-(\frac{1-K}{12[\mathcal{K}+1]\mathcal{M}}-\frac{1+2K}{24[\mathcal{K}+1]\mathcal{M}}\sqrt{\epsilon})\Omega(\delta)
+C\epsilon(1+\sqrt{\epsilon})\delta\notag\\
&+C\epsilon\frac{1+\sqrt{\epsilon}}{8[\mathcal{K}+1]}\Omega(\delta)
+C\epsilon_{0}\frac{1+\sqrt{\epsilon}}{8[\mathcal{K}+1]}\Omega(\delta)\notag\\
\leq&\frac{\hbar_{4}}{8[\mathcal{K}+1]}\Omega(\delta),\label{c81}
\end{align}
where the constant $\hbar_{4}$ satisfies $0<|K_{2}|\hbar_{4}<\frac{2+K}{3}$.

At the boundary $x=0$, similar to~\eqref{c76}, we have
\begin{align}
\varpi(\delta|\Psi_{2}^{(l)}(\cdot,0))\leq&(\frac{1-K}{24[\mathcal{K}+1]}+\frac{|K_{2}|\hbar_{4}}{8[\mathcal{K}+1]})\Omega(\delta)\notag\\
<&\frac{1}{8[\mathcal{K}+1]}\Omega(\delta),\label{c82}
\end{align}
then by the equations~\eqref{c59} and using~\eqref{c19}, \eqref{c29}, \eqref{c31}, \eqref{c34}, \eqref{c77} and~\eqref{c82}, we get
\begin{align}
\varpi(\delta|\Psi_{2}^{(l)}(\cdot,x))\leq&\Big((\frac{1-K}{24[\mathcal{K}+1]}+\frac{|K_{2}|\hbar_{4}}{8[\mathcal{K}+1]})
(1+\sqrt{\epsilon})\Omega(\delta)
+C\epsilon(1+\sqrt{\epsilon})\delta+C\epsilon\frac{1+\sqrt{\epsilon}}{8[\mathcal{K}+1]}\Omega(\delta)\Big)e^{C\epsilon L}\notag\\
\leq&\frac{1}{8[\mathcal{K}+1]}\Omega(\delta).\label{c83}
\end{align}
We complete the proof of the estimates~\eqref{c23}.

Finally, we prove~\eqref{c25}. We first consider the special case that two given points $(t_{1},x_{1})$ and $(t_{2},x_{2})$ with $|t_{1}-t_{2}|\leq\delta,|x_{1}-x_{2}|\leq\delta$ locate on the same characteristic curve $t=t_{1}^{(l)}(x;t_{0},x_{0})$, namely, $t_{2}=t_{1}^{(l)}(x_{2};t_{1},x_{1})$. Using the similar method of~\eqref{c64}, we can get
\begin{align}
|\Psi_{1}^{(l)}(t_{1},x_{1})-\Psi_{1}^{(l)}(t_{2},x_{2})|\leq C\epsilon\delta\leq\frac{1}{12}\Omega(\delta).\label{c84}
\end{align}
Then, for general two points $(t_{1},x_{1})$ and $(t_{2},x_{2})$ with $|t_{1}-t_{2}|\leq\delta,|x_{1}-x_{2}|\leq\delta$, we can choose a point $(t_{3},x_{1})$ locating on the $1$-st characteristic curve passing through $(t_{2},x_{2})$, namely, $t_{3}=t_{1}^{(l)}(x_{1};t_{2},x_{2})$.\\
\indent By~\eqref{B13} and definition~\eqref{c48}, we have
$$|t_{3}-t_{2}|\leq|\mu_{1}||x_{1}-x_{2}|\leq \mathcal{K}\delta,$$
and thus
$$|t_{3}-t_{1}|\leq|t_{3}-t_{2}|+|t_{2}-t_{1}|\leq(\mathcal{K}+1)\delta.$$
Now we combine estimates~\eqref{c81} and~\eqref{c84} to get
\begin{align}
&|\Psi_{1}^{(l)}(t_{1},x_{1})-\Psi_{1}^{(l)}(t_{2},x_{2})|\notag\\
\leq&|\Psi_{1}^{(l)}(t_{1},x_{1})-\Psi_{1}^{(l)}(\frac{[\mathcal{K}+1]t_{1}+t_{3}}{\mathcal{K}+1},x_{1})|\notag\\
&+|\Psi_{1}^{(l)}(\frac{[\mathcal{K}+1]t_{1}+t_{3}}{\mathcal{K}+1},x_{1})
-\Psi_{1}^{(l)}(\frac{([\mathcal{K}+1]-1)t_{1}+2t_{3}}{\mathcal{K}+1},x_{1})|\notag\\
&+\ldots+|\Psi_{1}^{(l)}(\frac{t_{1}+[\mathcal{K}+1]t_{3}}{\mathcal{K}+1},x_{1})-\Psi_{1}^{(l)}(t_{3},x_{1})|
+|\Psi_{1}^{(l)}(t_{3},x_{1})-\Psi_{1}^{(l)}(t_{2},x_{2})|\notag\\
\leq&\frac{[\mathcal{K}+1]+1}{8[\mathcal{K}+1]}\Omega(\delta)+\frac{1}{12}\Omega(\delta)\notag\\
\leq&\frac{1}{3}\Omega(\delta).\label{c85}
\end{align}
The combination of~\eqref{c84} and~\eqref{c85} leads to
\begin{align}
\varpi(\delta|\Psi_{1}^{(l)})\leq\frac{1}{3}\Omega(\delta).\label{c86}
\end{align}
With the aid of equations~\eqref{c4} and by~\eqref{a2}, \eqref{A3}, \eqref{B13}-\eqref{a4}, \eqref{c29}-\eqref{c31}, \eqref{c77} and~\eqref{c86}, we have
\begin{align}
\varpi(\delta|\partial_{x}\Phi_{1}^{(l)})\leq \mathcal{K}\frac{1}{2}\Omega(\delta).\label{c87}
\end{align}
Similarly, we get
\begin{align}
|\Psi_{2}^{(l)}(t_{1},x_{1})-\Psi_{2}^{(l)}(t_{2},x_{2})|\leq\frac{1}{12}\Omega(\delta)\label{c88}
\end{align}
for two given points $(t_{1},x_{1})$ and $(t_{2},x_{2})$ locate on the same characteristic curve $t=t_{1}^{(l)}(x;t_{0},x_{0})$ and
\begin{align}
|\Psi_{2}^{(l)}(t_{1},x_{1})-\Psi_{2}^{(l)}(t_{2},x_{2})|\leq&\frac{[\mathcal{K}+1]+1}{8[\mathcal{K}+1]}\Omega(\delta)
+\frac{1}{12}\Omega(\delta)\notag\\
\leq&\frac{1}{3}\Omega(\delta)\label{c89}
\end{align}
for general two points $(t_{1},x_{1})$ and $(t_{2},x_{2})$.

The combination of~\eqref{c88} and~\eqref{c89} leads to
\begin{align}
\varpi(\delta|\Psi_{2}^{(l)})\leq\frac{1}{3}\Omega(\delta).\label{c90}
\end{align}

Then using~\eqref{c29}-\eqref{c31}, \eqref{c77} and~\eqref{c90}, it follows from that the equtions~\eqref{c6} that
\begin{align}
\varpi(\delta|\partial_{x}\Phi_{2}^{(l)})\leq \mathcal{K}\frac{1}{2}\Omega(\delta).\label{c91}
\end{align}

Thus, by~\eqref{c86}-\eqref{c87} and~\eqref{c90}-\eqref{c91}, we get~\eqref{c25} holds directly.

\end{proof}
\indent With the help of~\proref{p1} and the similar arguments as in \cite{Qu}, the proof of~\theref{t1} could be presented, here we omit the details.
\section{\texorpdfstring{$C^{0}$}~~Stability of the Time-periodic Solution}\label{s4}
\indent\indent In this section, we give the proof of~\theref{t2} to consider the stability of the time-periodic solution obtained in~\theref{t1}. In order to prove the existence of classical solutions $\Phi=\Phi(t,x)$ to problem~\eqref{b8}-\eqref{b13}, we only need to prove the following ~\lemref{l1} on the basis of the existence and uniqueness of local $C^{1}$ solution for the mixed initial-boundary value problem for quasilinear hyperbolic system(cf. Chapter $4$ in \cite{Yuw}).
Using the method in~\cite{Li}, we can give the proof of~\lemref{l1}. Here we omit the details.
\begin{lemma}\label{l1}
When $\frac{\gamma-1}{4}(\tilde{\Phi}_{3}-\tilde{\Phi}_{1})<c_{-}<\epsilon_{0}$, there exists a small constant $\epsilon_{6}>0$ and the positive constant $C_{L}$, such that for any given $\epsilon\in(0,\epsilon_{6})$, there exist $\sigma=\sigma(\epsilon)>0$ and $M_{4}>0$ such that if
 \begin{align*}
 &\|H_{1}\|_{C^{1}(\mathbb{R}_{+})}\leq\sigma,~~\|H_{2}\|_{C^{0}(\mathbb{R}_{+})}\leq M_{4},~~\|\partial_{t}H_{2}\|_{C^{0}(\mathbb{R}_{+})}\leq\sigma,
 ~~\|H_{3}\|_{C^{1}(\mathbb{R}_{+})}\leq\sigma,\\
 &\|{\Phi}_{1_{0}}\|_{C^{1}[0,L]}\leq\sigma,~~\|{\Phi}_{2_{0}}\|_{C^{0}[0,L]}\leq M_{4},~~\|\partial_{x}{\Phi}_{2_{0}}\|_{C^{0}[0,L]}\leq\sigma,
 ~~\|{\Phi}_{3_{0}}\|_{C^{1}[0,L]}\leq\sigma,
 \end{align*}
 then the $C^{1}$ solution $\Phi=\Phi(t,x)$ to the initial-boundary value problem~\eqref{b8}-\eqref{b13} satisfies
\begin{align}
\|\Phi_{1}\|_{C^{1}(D)}\leq C_{L}\epsilon,~\|\Phi_{2}\|_{C^{0}(D)}\leq C_{L},~\max\{\|\partial_{t}\Phi_{2}\|_{C^{0}(D)},\|\partial_{x}\Phi_{2}\|_{C^{0}(D)}\}\leq C_{L}\epsilon,~\|\Phi_{3}\|_{C^{1}(D)}\leq C_{L}\epsilon. \label{d1}
\end{align}
\end{lemma}
Now, we prove~\eqref{b19} inductively. For any $t_{*}>0$ and $N\in\mathbb{N}$, we prove
\begin{align}
&\|\Phi_{1}(t,\cdot)-\Phi_{1}^{(P)}(t,\cdot)\|_{C^{0}}\leq C_{S}\epsilon\xi^{N+1},\quad\forall t\in[t_{*}+T_{0},t_{*}+2T_{0}],\label{D2}\\
&\|\Phi_{2}(t,\cdot)-\Phi_{2}^{(P)}(t,\cdot)\|_{C^{0}}\leq A_{0}C_{S}\epsilon\xi^{N},\quad\forall t\in[t_{*}+T_{0},t_{*}+2T_{0}],\label{d3}\\
&\|\Phi_{3}(t,\cdot)-\Phi_{3}^{(P)}(t,\cdot)\|_{C^{0}}\leq C_{S}\epsilon\xi^{N+1},\quad\forall t\in[t_{*}+T_{0},t_{*}+2T_{0}]\label{d2}
\end{align}
under the hypothesis
\begin{align}
&\|\Phi_{1}(t,\cdot)-\Phi_{1}^{(P)}(t,\cdot)\|_{C^{0}}\leq C_{S}\epsilon\xi^{N},\quad\quad\quad\forall t\in[t_{*},t_{*}+T_{0}],\label{D4}\\
&\|\Phi_{2}(t,\cdot)-\Phi_{2}^{(P)}(t,\cdot)\|_{C^{0}}\leq A_{0}C_{S}\epsilon\xi^{N-1},\quad\forall t\in[t_{*},t_{*}+T_{0}], \label{d5}\\
&\|\Phi_{3}(t,\cdot)-\Phi_{3}^{(P)}(t,\cdot)\|_{C^{0}}\leq C_{S}\epsilon\xi^{N},\quad\quad\quad\forall t\in[t_{*},t_{*}+T_{0}],\label{d4}
\end{align}
where $A_{0}>0$ and $\xi\in(0,1)$ is the constants will be determined later and $\Phi_{i}^{(P)}(t,x)(i=1,2,3)$ is the time-periodic solution obtained in~\theref{t1}.

At the boundary $x=L$, one has
$$
\Phi_{1}(t,L)-\Phi_{1}^{(P)}(t,L)=K_{1}(\Phi_{3}(t,L)-\Phi_{3}^{(P)}(t,L)),
$$
then from~\eqref{d4}, we have
\begin{align}
|\Phi_{1}(t,L)-\Phi_{1}^{(P)}(t,L)|\leq|K_{1}|C_{S}\epsilon\xi^{N}.\label{d6}
\end{align}
As for the interior estimates, we have
\begin{align}
&\partial_{x}\Phi_{1}^{(P)}+\mu_{1}(\Phi^{(P)}+\tilde{\Phi})\partial_{t}\Phi_{1}^{(P)}\notag\\
=
&\frac{\alpha}{2}(1-\frac{(\gamma+1)(\tilde{\Phi}_{1}+\tilde{\Phi}_{3})}{(\gamma+1)\tilde{\Phi}_{1}+(3-\gamma)\tilde{\Phi}_{3}})\mu_{1}
(\tilde{\Phi})\Phi_{1}^{(P)}\notag\\
&+\frac{\alpha}{2}(1-\frac{(\gamma+1)(\tilde{\Phi}_{1}+\tilde{\Phi}_{3})}{(\gamma+1)\tilde{\Phi}_{1}+(3-\gamma)\tilde{\Phi}_{3}})\mu_{1}
(\tilde{\Phi})\Phi_{3}^{(P)}\notag\\
&+\frac{\alpha(\gamma-1)(\tilde{\Phi}_{1}+\tilde{\Phi}_{3})}{(\gamma+1)\tilde{\Phi}_{1}+(3-\gamma)\tilde{\Phi}_{3}}
\mu_{1}(\Phi^{(P)}+\tilde{\Phi})\Phi_{3}^{(P)}\notag\\
&+\frac{\alpha}{2}(1-\frac{(\gamma+1)(\tilde{\Phi}_{1}+\tilde{\Phi}_{3})}
{(\gamma+1)\tilde{\Phi}_{1}+(3-\gamma)\tilde{\Phi}_{3}})\Big(\mu_{1}(\Phi^{(P)}+\tilde{\Phi})-\mu_{1}(\tilde{\Phi})\Big)(\Phi_{1}^{(P)}
+\Phi_{3}^{(P)})\notag\\
&+\frac{(\gamma-1)\varphi'(\Phi_{2}^{(P)}+\tilde{\Phi}_{2})}{16\gamma\varphi(\Phi_{2}^{(P)}+\tilde{\Phi}_{2})}(\Phi_{3}^{(P)}+\tilde{\Phi}_{3}-\Phi_{1}^{(P)}
-\tilde{\Phi}_{1})^{2}\mu_{1}(\Phi^{(P)}+\tilde{\Phi})\frac{\partial\Phi_{2}^{(P)}}{\partial x},\label{d7}
\end{align}
Then by~\eqref{c1} and~\eqref{d7}, we get
\begin{align}
&\partial_{x}(\Phi_{1}-\Phi_{1}^{(P)})+\mu_{1}(\Phi+\tilde{\Phi})\partial_{t}(\Phi_{1}-\Phi_{1}^{(P)})\notag\\
=&-\frac{\varphi'(\Phi_{2}^{(P)}+\tilde{\Phi}_{2})}{4\gamma\varphi(\Phi_{2}^{(P)}+\tilde{\Phi}_{2})}(\Phi_{3}^{(P)}+\tilde{\Phi}_{3}
-\Phi_{1}^{(P)}-\tilde{\Phi}_{1})\Big(\frac{\partial}{\partial x}(\Phi_{2}-\Phi_{2}^{(P)})+\mu_{1}(\Phi+\tilde{\Phi})\frac{\partial}{\partial t}(\Phi_{2}-\Phi_{2}^{(P)})\Big)\notag\\
&+\frac{\alpha}{2}(1-\frac{(\gamma+1)(\tilde{\Phi}_{1}+\tilde{\Phi}_{3})}{(\gamma+1)\tilde{\Phi}_{1}+(3-\gamma)\tilde{\Phi}_{3}})\mu_{1}
(\tilde{\Phi})(\Phi_{1}-\Phi_{1}^{(P)})\notag\\
&+\frac{\alpha}{2}(1-\frac{(\gamma+1)(\tilde{\Phi}_{1}+\tilde{\Phi}_{3})}{(\gamma+1)\tilde{\Phi}_{1}+(3-\gamma)\tilde{\Phi}_{3}})\mu_{1}
(\tilde{\Phi})(\Phi_{3}-\Phi_{3}^{(P)})\notag\\
&+\frac{\alpha(\gamma-1)(\tilde{\Phi}_{1}+\tilde{\Phi}_{3})}{(\gamma+1)\tilde{\Phi}_{1}+(3-\gamma)\tilde{\Phi}_{3}}\mu_{1}(\Phi
+\tilde{\Phi})(\Phi_{3}-\Phi_{3}^{(P)})\notag\\
&+\frac{(\gamma-1)\varphi'(\Phi_{2}^{(P)}+\tilde{\Phi}_{2})}{16\gamma\varphi(\Phi_{2}^{(P)}+\tilde{\Phi}_{2})}\Big((\Phi_{3}
-\Phi_{3}^{(P)})-(\Phi_{1}-\Phi_{1}^{(P)})\Big)(\Phi_{3}+\tilde{\Phi}_{3}-\Phi_{1}-\tilde{\Phi}_{1})\mu_{1}(\Phi+\tilde{\Phi})
\frac{\partial\Phi_{2}}{\partial x}\notag\\
&+\frac{\alpha}{2}(1-\frac{(\gamma+1)(\tilde{\Phi}_{1}+\tilde{\Phi}_{3})}{(\gamma+1)\tilde{\Phi}_{1}+(3-\gamma)\tilde{\Phi}_{3}})\Big
(\mu_{1}(\Phi+\tilde{\Phi})-\mu_{1}(\tilde{\Phi})\Big)\Big((\Phi_{1}+\Phi_{3})
-(\Phi_{1}^{(P)}+\Phi_{3}^{(P)})\Big)\notag\\
&+\frac{\gamma-1}{16\gamma}\Big(\frac{\varphi'(\Phi_{2}+\tilde{\Phi}_{2})}{\varphi(\Phi_{2}+\tilde{\Phi}_{2})}-\frac{\varphi'
(\Phi_{2}^{(P)}+\tilde{\Phi}_{2})}{\varphi(\Phi_{2}^{(P)}+\tilde{\Phi}_{2})}\Big)(\Phi_{3}+\tilde{\Phi}_{3}-\Phi_{1}-\tilde{\Phi}_{1})^{2}
\mu_{1}(\Phi+\tilde{\Phi})\frac{\partial\Phi_{2}}{\partial x}\notag\\
&-\frac{\varphi'(\Phi_{2}^{(P)}+\tilde{\Phi}_{2})}{4\gamma\varphi(\Phi_{2}^{(P)}+\tilde{\Phi}_{2})}(\Phi_{3}^{(P)}+\tilde{\Phi}_{3}
-\Phi_{1}^{(P)}-\tilde{\Phi}_{1})\Big(\mu_{1}(\Phi+\tilde{\Phi})-\mu_{1}(\Phi^{(P)}+\tilde{\Phi})\Big)\frac{\partial\Phi_{2}^{(P)}}
{\partial t}\notag\\
&+\frac{\alpha}{2}(1-\frac{(\gamma+1)(\tilde{\Phi}_{1}+\tilde{\Phi}_{3})}{(\gamma+1)\tilde{\Phi}_{1}+(3-\gamma)\tilde{\Phi}_{3}})\Big
(\mu_{1}(\Phi+\tilde{\Phi})-\mu_{1}(\Phi^{(P)}+\tilde{\Phi})\Big)\Phi_{1}^{(P)}\notag\\
&+\frac{\alpha}{2}(1-\frac{(3-\gamma)(\tilde{\Phi}_{1}+\tilde{\Phi}_{3})}{(\gamma+1)\tilde{\Phi}_{1}+(3-\gamma)\tilde{\Phi}_{3}})
\Big(\mu_{1}(\Phi+\tilde{\Phi})-\mu_{1}(\Phi^{(P)}+\tilde{\Phi})\Big)\Phi_{3}^{(P)}\notag\\
&-\Big(\mu_{1}(\Phi+\tilde{\Phi})-\mu_{1}(\Phi^{(P)}+\tilde{\Phi})\Big)\partial_{t}\Phi_{1}^{(P)},\label{d8}
\end{align}
where we used
\begin{align*}
&\frac{(\gamma-1)\varphi'(\Phi_{2}+\tilde{\Phi}_{2})}{16\gamma\varphi(\Phi_{2}+\tilde{\Phi}_{2})}(\Phi_{3}
+\tilde{\Phi}_{3}-\Phi_{1}-\tilde{\Phi}_{1})^{2}\mu_{1}(\Phi+\tilde{\Phi})\frac{\partial\Phi_{2}}{\partial x}\notag\\
=&-\frac{\varphi'(\Phi_{2}+\tilde{\Phi}_{2})}{4\gamma\varphi(\Phi_{2}+\tilde{\Phi}_{2})}(\Phi_{3}
+\tilde{\Phi}_{3}-\Phi_{1}-\tilde{\Phi}_{1})(\frac{\partial\Phi_{2}}{\partial x}+\mu_{1}(\Phi+\tilde{\Phi})
\frac{\partial\Phi_{2}}{\partial t}).
\end{align*}
We multiply $F_{1}(x)$ on both sides of~\eqref{d8} and integrate the result along the $1$-st characteristic curve $t=t_{1}(x;\hat{t},\hat{x})$ defined by
\begin{align}
\left\{
\begin{aligned}
&\frac{dt_{1}}{dx}(x;\hat{t},\hat{x})=\mu_{1}(\Phi(t_{1}(x;\hat{t},\hat{x}),x)+\tilde{\Phi}),\\
&t_{1}(\hat{x};\hat{t},\hat{x})=\hat{t}.
\end{aligned}\right.\label{d9}
\end{align}
Noting $T_{0}=L\mathop{\max}\limits_{i=1,2,3}\mathop{\sup}\limits_{\Phi\in \Re}|\mu_{i}(\Phi+\tilde{\Phi})|$, for each points $(\hat{t},\hat{x})\in[t_{*}+T_{0},\tau]\times[0,L]$, the backward curve $t=t_{1}(x;\hat{t},\hat{x})$ will intersect the boundary in a time interval shorter than $T_{0}$, namely,
$$t_{1}(L;\hat{t},\hat{x})\in[\hat{t}-T_{0},\hat{t}]\subseteq[t_{*},\tau],\quad \forall(\hat{t},\hat{x})\in[t_{*}+T_{0},\tau]\times[0,L],$$
and thus we can use estimates~\eqref{d6} on the boundary.

Using~\eqref{a2}, \eqref{A3}, \eqref{B13}-\eqref{a4}, \eqref{b18}-\eqref{B18}, \eqref{c42}, \eqref{d1} and~\eqref{d6}, we have
\begin{align}
&|\Phi_{1}(\hat{t},\hat{x})-\Phi_{1}^{(P)}(\hat{t},\hat{x})|\notag\\
\leq&\frac{|K_{1}|}{F_{1}(\hat{x})}C_{S}\epsilon\xi^{N}+\frac{F_{1}(\hat{x})-1}{F_{1}(\hat{x})}C_{S}\epsilon\xi^{N}
+C\epsilon_{0} C_{S}\epsilon\xi^{N}+C\epsilon C_{S}\epsilon\xi^{N}\notag\\
\leq&C_{S}\epsilon\xi^{N+1}.\label{d10}
\end{align}
\indent Here we choose a constant $0<\xi<1$ satisfying
$$ \xi>1-\frac{1-|K|}{\mathcal{M}}.$$
At the boundary $x=0$, by~\eqref{D4}, we get
\begin{align}
|\Phi_{2}(t,0)-\Phi_{2}^{(P)}(t,0)|\leq |K_{2}|C_{S}\epsilon\xi^{N}.\label{d11}
\end{align}
In the same way, we also get from~\eqref{B18}, \eqref{c2}, \eqref{D4}, \eqref{d4} and~\eqref{d11}
\begin{align}
|\Phi_{2}(\hat{t},\hat{x})-\Phi_{2}^{(P)}(\hat{t},\hat{x})| \leq& |K_{2}|C_{S}\epsilon\xi^{N}+C\epsilon C_{S}\epsilon\xi^{N}\notag\\
\leq&A_{0}C_{S}\epsilon\xi^{N},\label{d12}
\end{align}
where the constant $A_{0}$ satisfies $$A_{0}\geq|K_{2}|+C\epsilon.$$

Since $\hat{t}\in[t_{*}+T_{0},\tau]$ is arbitrary, we proved~\eqref{D2}-\eqref{d2}.

\section{Regularity of the Time-periodic Solution}\label{s5}
\indent\indent In this section, based on the higher regularity of boundary functions $H_{i}(t)(i=1,2,3)$, we will prove that the time-periodic solutions possess higher regularity.\\
\indent For the sake of showing the regularity of $\Phi^{(P)}$, we still employ the same iteration scheme~\eqref{c4}-\eqref{c9} as in~\secref{s3} and prove the following proposition.
\begin{proposition}\label{p2}
For the iteration scheme~\eqref{c4}-\eqref{c9}, assuming that~\eqref{b22} holds, then exists a large enough constant~$C_{R}>0$, such that for any given $l\in \mathbb{N}_{+}$, we have
\begin{align}
&\|\partial_{t}^{2}\Phi_{i}^{(l)}\|_{L^{\infty}}\leq C_{R},\label{e1}\\
&\|\partial_{t}\partial_{x}\Phi_{i}^{(l)}\|_{L^{\infty}}\leq \mathcal{K}C_{R},\label{e3}\\
&\|\partial_{x}^{2}\Phi_{i}^{(l)}\|_{L^{\infty}}\leq \mathcal{K}^{2}C_{R},\label{e5}
\end{align}
under the hypothesis
\begin{align}
&\|\partial_{t}^{2}\Phi_{i}^{(l-1)}\|_{L^{\infty}}\leq C_{R},\label{e7}\\
&\|\partial_{t}\partial_{x}\Phi_{i}^{(l-1)}\|_{L^{\infty}}\leq \mathcal{K}C_{R},\label{e9}\\
&\|\partial_{x}^{2}\Phi_{i}^{(l-1)}\|_{L^{\infty}}\leq \mathcal{K}^{2}C_{R}\label{e11}
\end{align}
with $i=1,2,3.$
\end{proposition}
\begin{proof}
\indent In fact, we utilize the same sequence as in~\secref{s3}. Then by~\proref{p1}, we already have~\eqref{c18}-\eqref{c25} for each $l$ and especially,
\begin{align}
&\|\Phi_{1}^{(l)}\|_{C^{1}}\leq(M_{1}+M_{2})\epsilon,\quad\|\Phi_{1}^{(l-1)}\|_{C^{1}}\leq(M_{1}+M_{2})\epsilon,\label{e13}\\
&\|\Phi_{2}^{(l)}\|_{C^{0}}\leq 2M_{0},~~\|\partial_{t}\Phi_{2}^{(l)}\|_{C^{0}}\leq M_{1}\epsilon,~~\|\partial_{x}\Phi_{2}^{(l)}\|_{C^{0}}\leq M_{2}\epsilon,\label{e14}\\
&\|\Phi_{2}^{(l-1)}\|_{C^{0}}\leq 2M_{0},~~\|\partial_{t}\Phi_{2}^{(l-1)}\|_{C^{0}}\leq M_{1}\epsilon,~~\|\partial_{x}\Phi_{2}^{(l-1)}\|_{C^{0}}\leq M_{2}\epsilon,\label{E14}\\
&\|\Phi_{3}^{(l)}\|_{C^{1}}\leq(M_{1}+M_{2})\epsilon,\quad\|\Phi_{3}^{(l-1)}\|_{C^{1}}\leq(M_{1}+M_{2})\epsilon.\label{e15}
\end{align}
Let
$$\Upsilon_{i}^{(l)}=\partial_{t}\Psi_{i}^{(l)}=\partial_{t}^{2}\Phi_{i}^{(l)},\quad i=1,2,3,~l\in \mathbb{N}_{+}.$$
Firstly, at the boundary $x=L$, using the boundary conditions~\eqref{c5} and the estimates~\eqref{b22}, \eqref{e7}, one has
\begin{align}
|\Upsilon_{1}^{(l)}(t,L)|\leq M_{0}+|K_{1}|C_{R}.\label{e16}
\end{align}
Then we use a method similar to~\eqref{c64} to get
\begin{align}
\|\Upsilon_{1}^{(l)}(t,x)\|_{L^{\infty}}\leq \chi_{1}C_{R},\label{e17}
\end{align}
where constants $0<\chi_{1}<1$ is independent of $l$.

By the equation~\eqref{c57} and the estimates~\eqref{a2}, \eqref{A3}, \eqref{B13}-\eqref{a4}, \eqref{e13}-\eqref{e15} and~\eqref{e17}, we get
\begin{align}
\|\partial_{x}\partial_{t}\Phi_{1}^{(l)}\|_{L^{\infty}}&\leq \mathcal{K}\chi_{1}C_{R}+C\epsilon\notag\\
&\leq \mathcal{K}\chi_{2}C_{R},\label{e18}
\end{align}
where $\chi_{2}\in(\chi_{1},1)$ is a constant independent of $l$.

Next, taking the spatial derivative to~\eqref{c4}, we have
\begin{align*}
\partial_{x}^{2}\Phi_{1}^{(l)}=&-\mu_{1}(\Phi^{(l-1)}+\tilde{\Phi})\partial_{x}\partial_{t}\Phi_{1}^{(l)}\notag\\
&+\frac{(\gamma-1)\varphi'(\Phi_{2}^{(l-1)}+\tilde{\Phi}_{2})}{16\gamma\varphi(\Phi_{2}^{(l-1)}+\tilde{\Phi}_{2})}
(\Phi_{3}^{(l-1)}+\tilde{\Phi}_{3}-\Phi_{1}^{(l-1)}
-\tilde{\Phi}_{1})^{2}\mu_{1}(\Phi^{(l-1)}+\tilde{\Phi})\partial_{x}^{2}\Phi_{2}^{(l-1)}\notag\\
&\frac{(\gamma-1)(\varphi''(\Phi_{2}^{(l-1)}+\tilde{\Phi}_{2})\varphi(\Phi_{2}^{(l-1)}+\tilde{\Phi}_{2})-\varphi'^{2}
(\Phi_{2}^{(l-1)}+\tilde{\Phi}_{2}))}{16\gamma\varphi^{2}(\Phi_{2}^{(l-1)}+\tilde{\Phi}_{2})}(\Phi_{3}^{(l-1)}+\tilde{\Phi}_{3}
-\Phi_{1}^{(l-1)}
-\tilde{\Phi}_{1})^{2}\notag\\
&\cdot\mu_{1}(\Phi^{(l-1)}+\tilde{\Phi})\partial_{x}\Phi_{2}^{(l-1)}(\partial_{x}\Phi_{2}^{(l-1)}+\tilde{\Phi}'_{2})\notag\\
&+\frac{(\gamma-1)\varphi'(\Phi_{2}^{(l-1)}+\tilde{\Phi}_{2})}{16\gamma\varphi(\Phi_{2}^{(l-1)}+\tilde{\Phi}_{2})}
(\Phi_{3}^{(l-1)}+\tilde{\Phi}_{3}-\Phi_{1}^{(l-1)}
-\tilde{\Phi}_{1})^{2}\Big(\frac{\partial\mu_{1}(\Phi^{(l-1)}+\tilde{\Phi})}{\partial\Phi_{1}^{(l-1)}}(\partial_{x}\Phi_{1}^{(l-1)}
+\tilde{\Phi}'_{1})\notag\\
&+\frac{\partial\mu_{1}(\Phi^{(l-1)}+\tilde{\Phi})}{\partial\Phi_{3}^{(l-1)}}(\partial_{x}\Phi_{3}^{(l-1)}
+\tilde{\Phi}'_{3})\Big)\partial_{x}\Phi_{2}^{(l-1)}\notag\\
&+\frac{(\gamma-1)\varphi'(\Phi_{2}^{(l-1)}+\tilde{\Phi}_{2})}{8\gamma\varphi(\Phi_{2}^{(l-1)}+\tilde{\Phi}_{2})}
(\Phi_{3}^{(l-1)}+\tilde{\Phi}_{3}-\Phi_{1}^{(l-1)}
-\tilde{\Phi}_{1})(\partial_{x}\Phi_{3}^{(l-1)}+\tilde{\Phi}'_{3}-\partial_{x}\Phi_{1}^{(l-1)}-\tilde{\Phi}'_{1})\notag\\
&\cdot\mu_{1}(\Phi^{(l-1)}+\tilde{\Phi})\partial_{x}\Phi_{2}^{(l-1)}\notag\\
&-\frac{\alpha^{2}(\gamma^{2}-1)(\tilde{\Phi}_{1}+\tilde{\Phi}_{3})}{2((\gamma+1)\tilde{\Phi}_{1}
+(3-\gamma)\tilde{\Phi}_{3})^{2}}(\mu_{3}(\tilde{\Phi})\tilde{\Phi}_{1}-\mu_{1}(\tilde{\Phi})\tilde{\Phi}_{3})(\mu_{1}(\Phi^{(l-1)}
+\tilde{\Phi})-\mu_{1}(\tilde{\Phi}))\Phi_{1}^{(l-1)}\notag\\
&-\Big(\frac{\partial\mu_{1}(\Phi^{(l-1)}+\tilde{\Phi})}{\partial\Phi_{1}^{(l-1)}}(\partial_{x}\Phi_{1}^{(l-1)}+\tilde{\Phi}'_{1})
+\frac{\partial\mu_{1}(\Phi^{(l-1)}+\tilde{\Phi})}{\partial\Phi_{3}^{(l-1)}}(\partial_{x}\Phi_{3}^{(l-1)}+\tilde{\Phi}'_{3})\Big)
\Psi_{1}^{(l)}\notag\\
&-\frac{\alpha^{2}(\gamma-1)(3-\gamma)(\tilde{\Phi}_{1}+\tilde{\Phi}_{3})}{2((\gamma+1)\tilde{\Phi}_{1}
+(3-\gamma)\tilde{\Phi}_{3})^{2}}(\mu_{3}(\tilde{\Phi})\tilde{\Phi}_{1}-\mu_{1}(\tilde{\Phi})\tilde{\Phi}_{3})\mu_{1}(\Phi^{(l-1)}
+\tilde{\Phi})\Phi_{3}^{(l-1)}\notag\\
&+\frac{\alpha'}{2}(1-\frac{(\gamma+1)(\tilde{\Phi}_{1}+\tilde{\Phi}_{3})}{(\gamma+1)\tilde{\Phi}_{1}+(3-\gamma)\tilde{\Phi}_{3}})
(\mu_{1}(\Phi^{(l-1)}+\tilde{\Phi})-\mu_{1}(\tilde{\Phi}))\Phi_{1}^{(l-1)}\notag\\
&+\frac{\alpha}{2}(1-\frac{(\gamma+1)(\tilde{\Phi}_{1}+\tilde{\Phi}_{3})}{(\gamma+1)\tilde{\Phi}_{1}+(3-\gamma)\tilde{\Phi}_{3}})
\Big(\frac{\partial\mu_{1}(\Phi^{(l-1)}+\tilde{\Phi})}{\partial\Phi_{1}^{(l-1)}}(\partial_{x}\Phi_{1}^{(l-1)}+\tilde{\Phi}'_{1})\notag\\
&+\frac{\partial\mu_{1}(\Phi^{(l-1)}+\tilde{\Phi})}{\partial\Phi_{3}^{(l-1)}}(\partial_{x}\Phi_{3}^{(l-1)}+\tilde{\Phi}'_{3})
-(\frac{\partial\mu_{1}(\tilde{\Phi})}{\partial\tilde{\Phi}_{1}}\tilde{\Phi}'_{1}+\frac{\partial\mu_{1}(\tilde{\Phi})}
{\partial\tilde{\Phi}_{3}}\tilde{\Phi}'_{3})\Big)\Phi_{1}^{(l-1)}\notag\\
&+\frac{\alpha}{2}(1-\frac{(\gamma+1)(\tilde{\Phi}_{1}+\tilde{\Phi}_{3})}{(\gamma+1)\tilde{\Phi}_{1}+(3-\gamma)\tilde{\Phi}_{3}})
(\mu_{1}(\Phi^{(l-1)}+\tilde{\Phi})-\mu_{1}(\tilde{\Phi}))\partial_{x}\Phi_{1}^{(l-1)}\notag\\
&+\frac{\alpha}{2}(1-\frac{(3-\gamma)(\tilde{\Phi}_{1}+\tilde{\Phi}_{3})}{(\gamma+1)\tilde{\Phi}_{1}+(3-\gamma)\tilde{\Phi}_{3}})
\Big(\frac{\partial\mu_{1}(\Phi^{(l-1)}+\tilde{\Phi})}{\partial\Phi_{1}^{(l-1)}}(\partial_{x}\Phi_{1}^{(l-1)}+\tilde{\Phi}'_{1})\notag\\
&+\frac{\partial\mu_{1}(\Phi^{(l-1)}+\tilde{\Phi})}{\partial\Phi_{3}^{(l-1)}}(\partial_{x}\Phi_{3}^{(l-1)}+\tilde{\Phi}'_{3})\Big)
\Phi_{3}^{(l-1)}\notag\\
&+\frac{\alpha}{2}(1-\frac{(\gamma+1)(\tilde{\Phi}_{1}+\tilde{\Phi}_{3})}{(\gamma+1)\tilde{\Phi}_{1}+(3-\gamma)\tilde{\Phi}_{3}})
(\frac{\partial\mu_{1}(\tilde{\Phi})}{\partial\tilde{\Phi}_{1}}\tilde{\Phi}'_{1}+\frac{\partial\mu_{1}(\tilde{\Phi})}
{\partial\tilde{\Phi}_{3}}\tilde{\Phi}'_{3})\Phi_{1}^{(l)}\notag\\
&-\frac{\alpha^{2}(\gamma^{2}-1)(\tilde{\Phi}_{1}+\tilde{\Phi}_{3})}{2((\gamma+1)\tilde{\Phi}_{1}
+(3-\gamma)\tilde{\Phi}_{3})^{2}}(\mu_{3}(\tilde{\Phi})\tilde{\Phi}_{1}-\mu_{1}(\tilde{\Phi})\tilde{\Phi}_{3})\mu_{1}(\tilde{\Phi})
\Phi_{1}^{(l)}\notag\\
&+\frac{\alpha}{2}(1-\frac{(3-\gamma)(\tilde{\Phi}_{1}+\tilde{\Phi}_{3})}{(\gamma+1)\tilde{\Phi}_{1}+(3-\gamma)\tilde{\Phi}_{3}})
\mu_{1}(\Phi^{(l-1)}+\tilde{\Phi})\partial_{x}\Phi_{3}^{(l-1)}\\
&+\frac{\alpha'}{2}(1-\frac{(3-\gamma)(\tilde{\Phi}_{1}+\tilde{\Phi}_{3})}{(\gamma+1)\tilde{\Phi}_{1}+(3-\gamma)\tilde{\Phi}_{3}})
\mu_{1}(\Phi^{(l-1)}+\tilde{\Phi})\Phi_{3}^{(l-1)}\notag\\
&+\frac{\alpha'}{2}(1-\frac{(\gamma+1)(\tilde{\Phi}_{1}+\tilde{\Phi}_{3})}{(\gamma+1)\tilde{\Phi}_{1}+(3-\gamma)\tilde{\Phi}_{3}})
\mu_{1}(\tilde{\Phi})\Phi_{1}^{(l)}\notag\\
&+\frac{\alpha}{2}(1-\frac{(\gamma+1)(\tilde{\Phi}_{1}+\tilde{\Phi}_{3})}
{(\gamma+1)\tilde{\Phi}_{1}+(3-\gamma)\tilde{\Phi}_{3}})\mu_{1}(\tilde{\Phi})\partial_{x}\Phi_{1}^{(l)}.\notag
\end{align*}
Then using~\eqref{a2}, \eqref{A3}, \eqref{B13}-\eqref{a4}, \eqref{e11}-\eqref{e15} and~\eqref{e18}, we get
\begin{align}
\|\partial_{x}^{2}\Phi_{1}^{(l)}\|_{L^{\infty}}\leq &\mathcal{K}^{2}\chi_{2}C_{R}+C\epsilon_{0}^{2}\mathcal{K}^{2}C_{R}+C\epsilon\notag\\
\leq&\mathcal{K}^{2}C_{R}.\label{e23}
\end{align}

At the boundary $x=0$, by~\eqref{b22}, \eqref{c7} and~\eqref{e17}, we have
\begin{align}
|\Upsilon_{2}^{(l)}(t,0)|\leq M_{0}+|K_{2}|\chi_{1}C_{R}.\label{e19}
\end{align}
Then we take the temporal derivative of~\eqref{c59} to get
\begin{align}
&\partial_{x}\Upsilon_{2}^{(l)}+\mu_{2}(\Phi^{(l-1)}+\tilde{\Phi})\partial_{t}\Upsilon_{2}^{(l)}\notag\\
=&-\Big(\frac{\partial^{2}\mu_{2}(\Phi^{(l-1)}+\tilde{\Phi})}{\partial(\Phi_{1}^{(l-1)})^{2}}(\Psi_{1}^{(l-1)})^{2}
+2\frac{\partial^{2}\mu_{2}
(\Phi^{(l-1)}+\tilde{\Phi})}{\partial\Phi_{1}^{(l-1)}\partial\Phi_{3}^{(l-1)}}\Psi_{1}^{(l-1)}\Psi_{3}^{(l-1)}
+\frac{\partial^{2}\mu_{2}(\Phi^{(l-1)}+\tilde{\Phi})}{\partial(\Phi_{3}^{(l-1)})^{2}}(\Psi_{3}^{(l-1)})^{2}\notag\\
&+\big(\frac{\partial\mu_{2}(\Phi^{(l-1)}+\tilde{\Phi})}{\partial\Phi_{1}^{(l-1)}}\Upsilon_{1}^{(l-1)}
+\frac{\partial\mu_{2}(\Phi^{(l-1)}+\tilde{\Phi})}{\partial\Phi_{3}^{(l-1)}}\Upsilon_{3}^{(l-1)}\big)\Big)\Psi_{2}^{(l)}\notag\\
&-\Big(\frac{\partial\mu_{2}(\Phi^{(l-1)}+\tilde{\Phi})}{\partial\Phi_{1}^{(l-1)}}\Psi_{1}^{(l-1)}
+\frac{\partial\mu_{2}(\Phi^{(l-1)}+\tilde{\Phi})}{\partial\Phi_{3}^{(l-1)}}\Psi_{3}^{(l-1)}\Big)\Upsilon_{2}^{(l)}.\label{e20}
\end{align}
With the aid of the Gronwall's inequality, \eqref{e7}, \eqref{e13}-\eqref{e15} and~\eqref{e19}, we get
\begin{align}
\|\Upsilon_{2}^{(l)}(t,x)\|_{L^{\infty}}\leq& \Big(M_{0}+|K_{2}|\chi_{1}C_{R}+C\epsilon^{3}+C\epsilon C_{R}\Big)e^{C\epsilon L}\notag\\
\leq&\chi_{3}C_{R},\label{e21}
\end{align}
where the constant $0<\chi_{1}<\chi_{3}<1$.

Using the equations~\eqref{c59} and the estimates~\eqref{B13}, \eqref{e13}-\eqref{e15}, one has
\begin{align}
\|\partial_{x}\partial_{t}\Phi_{2}^{(l)}\|_{L^{\infty}}&\leq \mathcal{K}\chi_{3}C_{R}+C\epsilon^{2}\notag\\
&\leq \mathcal{K}\chi_{4}C_{R},\label{e22}
\end{align}
where the constant $0<\chi_{3}<\chi_{4}<1$.

We take the spatial derivative of~\eqref{c6} to get
\begin{align*}
\partial_{x}^{2}\Phi_{2}^{(l)}=&-\Big(\frac{\partial\mu_{2}(\Phi^{(l-1)}+\tilde{\Phi})}{\partial\Phi_{1}^{(l-1)}}
(\partial_{x}\Phi_{1}^{(l-1)}+\tilde{\Phi}'_{1})
+\frac{\partial\mu_{2}(\Phi^{(l-1)}+\tilde{\Phi})}{\partial\Phi_{3}^{(l-1)}}(\partial_{x}\Phi_{3}^{(l-1)}+\tilde{\Phi}'_{3})\Big)
\Psi_{2}^{(l)}\notag\\
&-\mu_{2}(\Phi^{(l-1)}+\tilde{\Phi})\partial_{x}\partial_{t}\Phi_{2}^{(l)}
\end{align*}
and then by~\eqref{B13}, \eqref{e13}-\eqref{e15} and~\eqref{e22}, one has
\begin{align}
\|\partial_{x}^{2}\Phi_{2}^{(l)}\|_{L^{\infty}}\leq &\mathcal{K}^{2}\chi_{4}C_{R}+C\epsilon\notag\\
\leq&\mathcal{K}^{2}C_{R}.\label{e24}
\end{align}
We complete the proof of~\eqref{e1}-\eqref{e5}.

\end{proof}

$\mathbf{The~Proof~of~Theorem~2.3.}$ By~\eqref{e1}-\eqref{e5}, we know that $\{\Phi^{(l)}\}_{l=1}^{\infty}$ is uniformly $W^{2,\infty}$ bounded and then $weak^{*}$ convergent. Moreover, noting that $\{\Phi^{(l)}\}_{l=1}^{\infty}$ converges strongly to $\Phi^{(P)}$ in $C^{1}$, so we get the $W^{2,\infty}$ regularity of $\Phi^{(P)}$.
\section{ \texorpdfstring{$C^{1}$}~~Stability of the Time-periodic Solution}\label{s6}
\indent\indent In this section, we will give the proof of~\theref{t5}.

Actually, we have got the following $C^{0}$ exponential convergence in~\theref{t2}:
\begin{align}
&\|\Phi_{1}(t,\cdot)-\Phi_{1}^{(P)}(t,\cdot)\|_{C^{0}}\leq C_{S}\epsilon\xi^{N},
~\|\Phi_{3}(t,\cdot)-\Phi_{3}^{(P)}(t,\cdot)\|_{C^{0}}\leq C_{S}\epsilon\xi^{N},\label{f1}\\
&\|\Phi_{2}(t,\cdot)-\Phi_{2}^{(P)}(t,\cdot)\|_{C^{0}}\leq A_{0}C_{S}\epsilon\xi^{N-1},\quad\forall t\in[NT_{0},(N+1)T_{0}),\forall l\in \mathbb{N}_{+},\label{f2}
\end{align}
which show that
\begin{align}
&\|\Phi_{1}(t,\cdot)-\Phi_{1}^{(P)}(t,\cdot)\|_{C^{0}}\leq C_{S}\epsilon\xi^{N+1},
~\|\Phi_{3}(t,\cdot)-\Phi_{3}^{(P)}(t,\cdot)\|_{C^{0}}\leq C_{S}\epsilon\xi^{N+1},\label{f3}\\
&\|\Phi_{2}(t,\cdot)-\Phi_{2}^{(P)}(t,\cdot)\|_{C^{0}}\leq A_{0}C_{S}\epsilon\xi^{N},\quad\forall t\in[(N+1)T_{0},(N+2)T_{0}),\forall l\in \mathbb{N}_{+}.\label{f4}
\end{align}
Moreover, by~\theref{t1}, \theref{t4}, and~\lemref{l1}, we have
\begin{align}
&\|\Phi_{1}\|_{C^{1}}\leq C_{L}\epsilon,  ~\|\Phi_{3}\|_{C^{1}}\leq C_{L}\epsilon,\label{f5}\\
\|\Phi_{2}\|_{C^{0}(D)}\leq& C_{L},~\max\{\|\partial_{t}\Phi_{2}\|_{C^{0}(D)},\|\partial_{x}\Phi_{2}\|_{C^{0}(D)}\}\leq C_{L}\epsilon,\label{F5}
\end{align}
\begin{align}
&\|\Phi_{1}^{(P)}\|_{C^{1}}\leq C_{E}\epsilon, ~\|\Phi_{3}^{(P)}\|_{C^{1}}\leq C_{E}\epsilon,\label{f6}\\
\|\Phi_{2}^{(P)}\|_{C^{0}(D)}\leq & C_{E}M_{0},~\max\{\|\partial_{t}\Phi_{2}^{(P)}\|_{C^{0}(D)},\|\partial_{t}\Phi_{2}^{(P)}\|_{C^{0}(D)}\}\leq C_{E}\epsilon,\label{F6}
\end{align}
\begin{align}
\|\Phi^{(P)}\|_{W^{2,\infty}}\leq(1+\mathcal{K})^{2}C_{R}.\label{f7}
\end{align}
Noting the continuity, we will get the estimates for the convergence of the first derivatives inductively, namely, for each $N\in \mathbb{N}_{+}$ and $\tau\in[(N+1)T_{0},(N+2)T_{0}]$, we will prove
\begin{align}
&\|\partial_{t}\Phi_{1}(t,\cdot)-\partial_{t}\Phi_{1}^{(P)}(t,\cdot)\|_{C^{0}}\leq C_{S}^{*}\epsilon\xi^{N+1},
~\|\partial_{t}\Phi_{3}(t,\cdot)-\partial_{t}\Phi_{3}^{(P)}(t,\cdot)\|_{C^{0}}\leq C_{S}^{*}\epsilon\xi^{N+1},\label{f8}\\
&\|\partial_{t}\Phi_{2}(t,\cdot)-\partial_{t}\Phi_{2}^{(P)}(t,\cdot)\|_{C^{0}}\leq A_{1}C_{S}^{*}\epsilon\xi^{N},\quad\forall t\in[(N+1)T_{0},\tau],\forall l\in \mathbb{N}_{+},\label{f9}\\
&\|\partial_{x}\Phi_{1}(t,\cdot)-\partial_{x}\Phi_{1}^{(P)}(t,\cdot)\|_{C^{0}}\leq \mathcal{K}C_{S}^{*}\epsilon\xi^{N+1},
~\|\partial_{x}\Phi_{3}(t,\cdot)-\partial_{x}\Phi_{3}^{(P)}(t,\cdot)\|_{C^{0}}\leq \mathcal{K}C_{S}^{*}\epsilon\xi^{N+1},\label{f10}\\
&\|\partial_{x}\Phi_{2}(t,\cdot)-\partial_{x}\Phi_{2}^{(P)}(t,\cdot)\|_{C^{0}}\leq \mathcal{K}A_{1}C_{S}^{*}\epsilon\xi^{N},\quad\forall t\in[(N+1)T_{0},\tau],\forall l\in \mathbb{N}_{+}\label{f11}
\end{align}
under the assumption
\begin{align}
&\|\partial_{t}\Phi_{1}(t,\cdot)-\partial_{t}\Phi_{1}^{(P)}(t,\cdot)\|_{C^{0}}\leq C_{S}^{*}\epsilon\xi^{N},
~\|\partial_{t}\Phi_{3}(t,\cdot)-\partial_{t}\Phi_{3}^{(P)}(t,\cdot)\|_{C^{0}}\leq C_{S}^{*}\epsilon\xi^{N},\label{f12}\\
&\|\partial_{t}\Phi_{2}(t,\cdot)-\partial_{t}\Phi_{2}^{(P)}(t,\cdot)\|_{C^{0}}\leq A_{1}C_{S}^{*}\epsilon\xi^{N-1},\quad\forall t\in[NT_{0},\tau],\forall l\in \mathbb{N}_{+},\label{f13}\\
&\|\partial_{x}\Phi_{1}(t,\cdot)-\partial_{x}\Phi_{1}^{(P)}(t,\cdot)\|_{C^{0}}\leq \mathcal{K}C_{S}^{*}\epsilon\xi^{N},
~\|\partial_{x}\Phi_{3}(t,\cdot)-\partial_{x}\Phi_{3}^{(P)}(t,\cdot)\|_{C^{0}}\leq \mathcal{K}C_{S}^{*}\epsilon\xi^{N},\label{f14}\\
&\|\partial_{x}\Phi_{2}(t,\cdot)-\partial_{x}\Phi_{2}^{(P)}(t,\cdot)\|_{C^{0}}\leq \mathcal{K}A_{1}C_{S}^{*}\epsilon\xi^{N-1},\quad\forall t\in[NT_{0},\tau],\forall l\in \mathbb{N}_{+},\label{f15}
\end{align}
where the constant $A_{1}>0$ will be determined later.

Let
$$\Psi_{i}=\partial_{t}\Phi_{i},\quad \Xi_{i}=\partial_{x}\Phi_{i},\quad i=1,2,3$$
and
$$\Psi_{i}^{(P)}=\partial_{t}\Phi_{i}^{(P)},\quad \Xi_{i}^{(P)}=\partial_{x}\Phi_{i}^{(P)},i=1,2,3.$$
Taking the temporal derivative on boundary conditions~\eqref{b12}-\eqref{b13}, we get
\begin{align}
&\Psi_{1}(t,L)=H'_{1}(t)+K_{1}\Psi_{3}(t,L), \quad t>0,\label{f16}\\
&\Psi_{2}(t,0)=H'_{2}(t)+K_{2}\Psi_{1}(t,0),\quad~ t>0,\label{f17}\\
&\Psi_{3}(t,0)=H'_{3}(t)+K_{3}\Psi_{1}(t,0),\quad~ t>0 \label{f18}
\end{align}
and
\begin{align}
&\Psi_{1}^{(P)}(t,L)=H'_{1}(t)+K_{1}\Psi_{3}^{(P)}(t,L),\quad t>0,\label{f19}\\
&\Psi_{2}^{(P)}(t,0)=H'_{2}(t)+K_{2}\Psi_{1}^{(P)}(t,0),\quad~ t>0,\label{f20}\\
&\Psi_{3}^{(P)}(t,0)=H'_{3}(t)+K_{3}\Psi_{1}^{(P)}(t,0),\quad~t>0. \label{f21}
\end{align}
Then, at the boundary $x=L$, it follows from~\eqref{f12}
\begin{align}
\mathop{\sup}\limits_{t\in[NT_{0},\tau]}|\Psi_{1}(t,L)-\Psi_{1}^{(P)}(t,L)|\leq|K_{1}|C_{S}^{*}\epsilon\xi^{N}.\label{f22}
\end{align}

In the domain $D$, similar to the proof method of~\eqref{d10}, we get
\begin{align}
|\Psi_{1}(\hat{t},\hat{x})-\Psi_{1}^{(P)}(\hat{t},\hat{x})|\leq&\frac{|K_{1}|C_{S}^{*}\epsilon\xi^{N}}{F_{1}(\hat{x})}+C\epsilon C_{S}\epsilon\xi^{N}+C\epsilon C_{S}^{*}\epsilon\xi^{N}
+\frac{F_{1}(\hat{x})-1}{F_{1}(\hat{x})}C_{S}^{*}\epsilon\xi^{N}\notag\\
&+\mathop{\sup}\limits_{\Phi\in\Re}|\nabla\mu_{1}|(1+\mathcal{K})^{2}C_{R}C_{S}\epsilon\xi^{N}
+C\epsilon_{0} C_{S}^{*}\epsilon\xi^{N}\notag\\
\leq&\chi_{5}C_{S}^{*}\epsilon\xi^{N+1},\label{f25}
\end{align}
where the constant $0<\chi_{5}<1$ and we choose the constant $C_{S}^{*}$ satisfying
$$C_{S}^{*}>100\mathop{\max}\limits_{1\leq i\leq3}\mathop{\sup}\limits_{\Phi\in\Re}|\nabla\mu_{i}|(1+\mathcal{K})^{2}C_{R}C_{S}+C_{S}
-100\alpha_{*}(1+\frac{\gamma+1}{2}\mathcal{K}c_{-})\mathcal{K}C_{S}.$$

At the boundary $x=0$, it follows from~\eqref{f12}
\begin{align}
&\mathop{\sup}\limits_{t\in[NT_{0},\tau]}|\Psi_{2}(t,0)-\Psi_{2}^{(P)}(t,0)|\leq |K_{2}|C_{S}^{*}\epsilon
\xi^{N}.\label{f23}
\end{align}
We take the temporal derivative of~\eqref{c2} and subtract it from the equations of $\Psi_{2}^{(P)}$ to obtain
\begin{align}
&\partial_{x}(\Psi_{2}-\Psi_{2}^{(P)})+\mu_{2}(\Phi+\tilde{\Phi})\partial_{t}(\Psi_{2}-\Psi_{2}^{(P)})\notag\\
=&-\Big(\big(\frac{\partial\mu_{2}(\Phi+\tilde{\Phi})}{\partial\Phi_{1}}-\frac{\partial\mu_{2}(\Phi^{(P)}+\tilde{\Phi})}
{\partial\Phi_{1}^{(P)}}\big)\Psi_{1}^{(P)}+\big(\frac{\partial\mu_{2}(\Phi+\tilde{\Phi})}{\partial\Phi_{3}}
-\frac{\partial\mu_{2}(\Phi^{(P)}+\tilde{\Phi})}{\partial\Phi_{3}^{(P)}}\big)\Psi_{3}^{(P)}\Big)\Psi_{2}^{(P)}\notag\\
&-\Big(\frac{\partial\mu_{2}(\Phi+\tilde{\Phi})}{\partial\Phi_{1}}(\Psi_{1}-\Psi_{1}^{(P)})
+\frac{\partial\mu_{2}(\Phi+\tilde{\Phi})}{\partial\Phi_{3}}(\Psi_{3}-\Psi_{3}^{(P)})\Big)\Psi_{2}^{(P)}\notag\\
&-\Big(\frac{\partial\mu_{2}(\Phi+\tilde{\Phi})}{\partial\Phi_{1}}\Psi_{1}+\frac{\partial\mu_{2}(\Phi
+\tilde{\Phi})}{\partial\Phi_{3}}\Psi_{3}\Big)(\Psi_{2}-\Psi_{2}^{(P)})\notag\\
&-\Big(\mu_{2}(\Phi+\tilde{\Phi})-\mu_{2}(\Phi^{(P)}+\tilde{\Phi})\Big)\partial_{t}\Psi_{2}^{(P)},\label{f26}
\end{align}
then by~\eqref{f1}, \eqref{f5}-\eqref{f7}, \eqref{f12}-\eqref{f13} and~\eqref{f23}, we get
\begin{align}
&|\Psi_{2}(\hat{t},\hat{x})-\Psi_{2}^{(P)}(\hat{t},\hat{x})|\notag\\
\leq& |K_{2}|C_{S}^{*}\epsilon\xi^{N}+\mathop{\sup}\limits_{\Phi\in\Re}|\nabla\mu_{2}|(1+\mathcal{K})^{2}
C_{R}C_{S}\epsilon\xi^{N}\notag\\
&+C\epsilon C_{S}\epsilon\xi^{N}+C\epsilon C_{S}^{*}\epsilon\xi^{N}\notag\\
\leq&\chi_{6}A_{1}C_{S}^{*}\epsilon\xi^{N},\label{f27}
\end{align}
where the constant $0<\chi_{6}<1$ and we choose the constant $A_{1}$ satisfying $$\chi_{6}A_{1}\geq|K_{2}|+\frac{1}{100}+C\epsilon.$$

By the arbitrariness of $\hat{t}\in[(N+1)T_{0},\tau]$, we obtain
\begin{align}
&\|\Psi_{1}(t,\cdot)-\Psi_{1}^{(P)}(t,\cdot)\|_{C^{0}}\leq \chi_{5}C_{S}^{*}\epsilon\xi^{N+1},\quad\forall t\in[(N+1)T_{0},\tau],\forall l\in \mathbb{N}_{+},\label{f28}\\
&\|\Psi_{2}(t,\cdot)-\Psi_{2}^{(P)}(t,\cdot)\|_{C^{0}}\leq \chi_{6}A_{1}C_{S}^{*}\epsilon\xi^{N},\quad\forall t\in[(N+1)T_{0},\tau],\forall l\in \mathbb{N}_{+}.\label{f29}
\end{align}

With the aid of~\eqref{c1} and~\eqref{d7}, we get
\begin{align}
\Xi_{1}-\Xi_{1}^{(P)}=&-\mu_{1}(\Phi^{(P)}+\tilde{\Phi})
(\Psi_{1}-\Psi_{1}^{(P)})\notag\\
&+\frac{(\gamma-1)\varphi'(\Phi_{2}^{(P)}+\tilde{\Phi}_{2})}{16\gamma\varphi(\Phi_{2}^{(P)}+\tilde{\Phi}_{2})}
(\Phi_{3}^{(P)}+\tilde{\Phi}_{3}-\Phi_{1}^{(P)}-\tilde{\Phi}_{1})^{2}\mu_{1}(\Phi^{(P)}
+\tilde{\Phi})(\frac{\partial\Phi_{2}}{\partial x}-\frac{\partial\Phi_{2}^{(P)}}{\partial x})\notag\\
&+\frac{\alpha}{2}(1-\frac{(\gamma+1)(\tilde{\Phi}_{1}+\tilde{\Phi}_{3})}{(\gamma+1)\tilde{\Phi}_{1}+(3-\gamma)\tilde{\Phi}_{3}})
\mu_{1}(\tilde{\Phi})(\Phi_{1}+\Phi_{3}-\Phi_{1}^{(P)}-\Phi_{3}^{(P)})\notag\\
&+\frac{\alpha(\gamma-1)(\tilde{\Phi}_{1}+\tilde{\Phi}_{3})}{(\gamma+1)\tilde{\Phi}_{1}+(3-\gamma)\tilde{\Phi}_{3}}\mu_{1}(\Phi^{(P)}
+\tilde{\Phi})(\Phi_{3}-\Phi_{3}^{(P)})\notag\\
&+\frac{\alpha}{2}(1-\frac{(\gamma+1)(\tilde{\Phi}_{1}+\tilde{\Phi}_{3})}{(\gamma+1)\tilde{\Phi}_{1}+(3-\gamma)\tilde{\Phi}_{3}})
\Big(\mu_{1}(\Phi^{(P)}+\tilde{\Phi})-\mu_{1}(\tilde{\Phi})\Big)(\Phi_{1}+\Phi_{3}-\Phi_{1}^{(P)}-\Phi_{3}^{(P)})\notag\\
&+\frac{(\gamma-1)\varphi'(\Phi_{2}^{(P)}+\tilde{\Phi}_{2})}{16\gamma\varphi(\Phi_{2}^{(P)}+\tilde{\Phi}_{2})}
(\Phi_{3}^{(P)}+\tilde{\Phi}_{3}-\Phi_{1}^{(P)}-\tilde{\Phi}_{1})^{2}\Big(\mu_{1}(\Phi+
\tilde{\Phi})-\mu_{1}(\Phi^{(P)}+\tilde{\Phi})\Big)\frac{\partial\Phi_{2}}{\partial x}\notag\\
&+\frac{\gamma-1}{16\gamma}(\frac{\varphi'(\Phi_{2}+\tilde{\Phi}_{2})}{\varphi(\Phi_{2}+\tilde{\Phi}_{2})}
-\frac{\varphi'(\Phi_{2}^{(P)}+\tilde{\Phi}_{2})}{\varphi(\Phi_{2}^{(P)}+\tilde{\Phi}_{2})})(\Phi_{3}+\tilde{\Phi}_{3}-\Phi_{1}
-\tilde{\Phi}_{1})^{2}\mu_{1}(\Phi+\tilde{\Phi})\frac{\partial\Phi_{2}}{\partial x}\notag\\
&+\frac{(\gamma-1)\varphi'(\Phi_{2}^{(P)}+\tilde{\Phi}_{2})}{16\gamma\varphi(\Phi_{2}^{(P)}+\tilde{\Phi}_{2})}
\Big((\Phi_{3}+\tilde{\Phi}_{3}-\Phi_{1}-\tilde{\Phi}_{1})^{2}-(\Phi_{3}^{(P)}
+\tilde{\Phi}_{3}-\Phi_{1}^{(P)}-\tilde{\Phi}_{1})^{2}\Big)\notag\\
&\cdot\mu_{1}(\Phi+\tilde{\Phi})\frac{\partial\Phi_{2}}{\partial x}\notag\\
&+\frac{\alpha}{2}(1-\frac{(\gamma+1)(\tilde{\Phi}_{1}+\tilde{\Phi}_{3})}{(\gamma+1)\tilde{\Phi}_{1}+(3-\gamma)\tilde{\Phi}_{3}})
\Big(\mu_{1}(\Phi+\tilde{\Phi})-\mu_{1}(\Phi^{(P)}+\tilde{\Phi})\Big)(\Phi_{1}+\Phi_{3})\notag\\
&+\frac{\alpha(\gamma-1)(\tilde{\Phi}_{1}+\tilde{\Phi}_{3})}{(\gamma+1)\tilde{\Phi}_{1}+(3-\gamma)\tilde{\Phi}_{3}}\Big(\mu_{1}(\Phi+\tilde{\Phi})
-\mu_{1}(\Phi^{(P)}+\tilde{\Phi})\Big)\Phi_{3}\notag\\
&-\Big(\mu_{1}(\Phi+\tilde{\Phi})-\mu_{1}(\Phi^{(P)}+\tilde{\Phi})\Big)\Psi_{1}\notag,
\end{align}
then by~\eqref{a2}, \eqref{A3}, \eqref{B13}-\eqref{a4}, \eqref{f3}, \eqref{f5}-\eqref{f6}, \eqref{f15} and~\eqref{f28}, we have
\begin{align}
\|\Xi_{1}(t,\cdot)-\Xi_{1}^{(P)}(t,\cdot)\|_{C^{0}}\leq&\mathcal{K}\chi_{5}C_{S}^{*}\epsilon\xi^{N+1}
+C\epsilon_{0}^{2}\mathcal{K}A_{1}C_{S}^{*}\epsilon\xi^{N}
-\alpha_{*}(1+\frac{\gamma+1}{2}\mathcal{K}c_{-})\mathcal{K}C_{S}\epsilon\xi^{N+1}\notag\\
&+C\epsilon_{0}C_{S}\epsilon\xi^{N+1}+C\epsilon C_{S}\epsilon\xi^{N+1}\notag\\
\leq&\mathcal{K}C_{S}^{*}\epsilon\xi^{N+1}.\label{f30}
\end{align}
It follows from the equations~\eqref{c2}
\begin{align*}
\Xi_{2}-\Xi_{2}^{(P)}=-\Big(\mu_{2}(\Phi+\tilde{\Phi})-\mu_{2}(\Phi^{(P)}+\tilde{\Phi})\Big)\Psi_{2}^{(P)}-\mu_{2}(\Phi+\tilde{\Phi})
(\Psi_{2}-\Psi_{2}^{(P)}),
\end{align*}
and using~\eqref{B13}, \eqref{f3}, \eqref{F6} and~\eqref{f29}, we have
\begin{align}
\|\Xi_{2}(t,\cdot)-\Xi_{2}^{(P)}(t,\cdot)\|_{C^{0}}\leq& \mathcal{K}\chi_{6}A_{1}C_{S}^{*}\epsilon\xi^{N}+C\epsilon C_{S}\epsilon\xi^{N+1}\notag\\
\leq&\mathcal{K}A_{1}C_{S}^{*}\epsilon\xi^{N}.\label{f31}
\end{align}
This indicates that we complete the proof of~\eqref{f8}-\eqref{f11}.

\section{Acknowledge}\label{s7}

 \indent\indent Peng Qu is supported in part by the National Natural Science Foundation of China Grants No.~12122104, 11831011, and Shanghai Science and Technology Programs 21ZR1406000, 21JC1400600. Huimin Yu is supported in part by the National Natural Science Foundation of China Grant No. 12271310 and Natural Science Foundation of Shandong Province ZR2022MA088.

\end{sloppypar}
\end{document}